\documentclass[reqno]{amsart}
\usepackage[notref,notcite]{showkeys}

\begin{document}
\title
{Solutions to Yang-Mills equations}

\author
{Jorma Jormakka}

\address{Jorma Jormakka \newline
Karhekuja 4, 01660 Vantaa, Finland}
\email{jorma.o.jormakka@kolumbus.fi}

\keywords{Gauge field theory; Yang-Mills fields;
theoretical physics}

\begin{abstract}
 This article gives explicit solutions to the Yang-Mills equations. The 
solutions have positive energy that can be made arbitrarily small 
by selection
of a parameter showing that quantum Yang-Mills field theories do not
have a mass gap. 
\end{abstract}

\maketitle
\numberwithin{equation}{section}
\newtheorem{theorem}{Theorem}[section]
\newtheorem{lemma}[theorem]{Lemma}
\allowdisplaybreaks

\section{Introduction}

In the year 2000 the Clay Mathematics Institute (CMI) posed the following problem [1]:

{\it Yang-Mills Existence and Mass Gap. Prove that for any compact simple group G, a nontrivial Yang-Mills theory exists on $\mathbb{R}^4$ and has a mass gap $\Delta>0$. Existence includes establishing axiomatic properties at least as strong as those cited in R. Streamer and A. Wightman (1964) or K. Osterwalden and R. Seiler (1973).}

Thus, the existence of a non-trivial Yang-Mills theory involves showing that
the theory fills
axioms of axiomatic quantum field theory, while the existence of a mass 
gap seems to be another question that may be shown also in some other way.
The problem concerns a pure Yang-Mills Lagrangian, i.e., only the gauge field
without spinor fields, Higgs fields, or other fields. The mass gap is expected
to arise from self-intersections of the Yang-Mills gauge field. The issue how
the mass gap could appear is unclear but [2] has proposed one mechanism.
The state of research to this problem up to 2004 is summarized in [3].
After that there have been some efforts to prove the existence of a mass gap,
e.g. [4], [5], but the problem is still considered open.

This article presents explicit solutions to the Yang-Mills Euler-Lagrange
equations. The solutions give arbitrarily small positive values for energy. 
This shows that the Hamiltonian has
arbitrarily small eigenvalues indicating that there is no mass gap.  
The solutions can be given on $\mathbb{R}^4$ with Minkowski's 
or Euclidean metric and they are simple, natural solutions that should be 
accepted as gauge fields in any non-trivial quantum field theory for the
pure Yang-Mills Lagrangian.    
  
\section {Definitions and notations}

We will first describe the problem setting as it can be presented in physics
in tensor calculus, and at the end look at the more mathematical 
formulation with differential forms and the Hodge star operator. Unless
otherwise stated, or the sum is written explicitely, there is summation over
indices that are repeated on one side of an equation. For notations we refer 
to [6].
    
\begin{equation}
\mathcal{L}=-\frac 12 Tr (F_{\mu\nu}F^{\mu\nu})\\
=-\frac 14 F_a^{\mu\nu}F^a_{\mu\nu}
\end{equation}
where
\begin{equation}
F^{\mu\nu}=F_a^{\mu\nu}t_a
\end{equation}
and $t_a$ are the generators of the Lie group satisfying
\begin{equation}
Tr(t_at_b)=\frac 12 \delta_{ab} \quad [t_b,t_c]=i f_{abc}t_a
\end{equation}
The structure constants $f_{abc}=f^{abc}$ are selected antisymmetric 
in all indices. The gauge field
\begin{equation}
A^{\mu}=A_a^{\mu}t_a
\end{equation}
defines the curvature $F^{\mu\nu}$ by
\begin{equation}
F^{\mu\nu}=\partial^{\mu}A^{\nu}-\partial^{\nu}A^{\mu}+ig[A^{\mu},A^{\nu}]\\
\end{equation}
In component form this gives
\begin{equation}
F^{\mu\nu}_a=\partial^{\mu}A_a^{\nu}-\partial^{\nu}A_a^{\mu}-gf_{abc}A^{\mu}_bA^{\nu}_c
\end{equation}
Curvature is antisymmetric
\begin{equation}
F^{\mu\nu}=-F^{\nu\mu}
\end{equation}
The number $g$ is called coupling constant, and
\begin{equation}
\partial^{\mu}=\frac {\partial}{\partial x_{\mu}}
\quad \partial_{\mu}=\frac {\partial}{\partial x^{\mu}}
\end{equation} 
are partial derivatives with respect to the contravariant coordinates
$x^{\mu}$ and contravariant coordinates $x_{\mu}=g_{\mu\nu}x^{\nu}$.
$x^0=ct$ and $x^j$, $1\le j\le 3$, are the space coordinates.
The metric $g_{\mu\nu}=g^{\mu\nu}$ is Minkowski's metric
\begin{equation}
\begin{gathered}
(g_{\mu\nu})_{\mu,\nu}=\left( \begin{matrix}
1 & 0 & 0 & 0\cr
0 & -1 & 0 & 0\cr
0 & 0 & -1 & 0\cr
0 & 0 & 0 & -1
\end{matrix} \right)
\end{gathered}
\end{equation}   
Thus $x_0=x^0$, $x_j=-x^j$ for $1\le j\le 3$. For real vectors and tensors
lowering and raising indices is made by
\begin{equation}
A_{\mu}^a=g_{\mu\nu}A^{\nu}_a\\
\quad F_{\mu\nu}^a=g_{\mu\alpha}g_{\nu\beta}F^{\alpha\beta}_a\\
\quad \partial_{\mu}=g_{\mu\nu}\partial^{\nu}
\end{equation}   
Therefore (2.6) can also be expressed as
\begin{gather*}
g_{\mu\alpha}g_{\nu\beta}
F^{\alpha\beta}_a=g_{\mu\alpha}g_{\nu\beta}\left( \partial ^{\alpha}A_a^{\beta}-\partial ^{\alpha}A_a^{\beta}-gf_{abc}A^{\alpha}_bA^{\beta}_c\right)
\end{gather*}
\begin{equation}
F_{\mu\nu}^a =\partial _{\mu}A^a_{\nu}-\partial_{\nu}A^a_{\mu}-gf^{abc}A_{\mu}^bA_{\nu}^c
\end{equation}   
where we have written $f^{abc}$ instead of $f_{abc}$ to follow the summation 
convention for the index $a$. The natural setting of quantum field theories
is that the component functions of the fields take complex values. Then raising
and lowering indices involves taking complex conjugates but we will only 
do calculations with real fields. Complex fields are better threated by
the algebraic geometric formulation described briefly at the end of 
Section 2.  

Let us notice that there is a summation over $b$ and $c$ in (2.6) and (2.11). 
We formulate this simple observation as a lemma since it is needed in the
sequence.
\begin{lemma} 
Let $c>b$. The last term in (2.6) can be expressed as 
\begin{equation}
f_{abc}A^{\mu}_bA^{\nu}_c=\sum_{c>b}f_{abc}\left(A^{\mu}_bA^{\nu}_c\\
-A^{\mu}_cA^{\nu}_b\right)
\end{equation}   
\end{lemma}

\begin{proof}
Expanding the commutator $[A^{\mu},A^{\nu}]$
\begin{equation}
\begin{gathered}
\left( \sum_b A^{\mu}_b t_b\right)\left( \sum_c A^{\nu}_ct_c\right)
-\left( \sum_c A^{\nu}_c t_c\right)\left( \sum_b A^{\mu}_bt_b\right)\\
=\sum_{b,c}\left(A^{\mu}_bA^{\nu}_ct_bt_c-A^{\nu}_cA^{\mu}_bt_ct_b\right)
\end{gathered}
\end{equation}  
Since $A^{\mu}_a$ are scalars $A^{\nu}_cA^{\mu}_b=A^{\mu}_bA^{\nu}_c$. Thus
we get
\begin{equation}
\begin{gathered}
[A^{\mu},A^{\nu}]=\sum_{b,c}A^{\mu}_bA^{\nu}_c[t_b,t_c]=\sum_{b,c}A^{\mu}_bA^{\nu}_cif_{abc}t_a\\
=\sum_{a,b,c}i\left(f_{abc}A^{\mu}_bA^{\nu}_c+f_{acb}A^{\mu}_cA^{\nu}_b\right)t_a\\
=\sum_a\sum_{c>b}if_{abc}\left(A^{\mu}_bA^{\nu}_c-A^{\mu}_cA^{\nu}_b\right)t_a
\end{gathered}
\end{equation}  
\end{proof}

As an example, let the group be SU(2). It has three generators
\begin{gather*}
t_1=\frac 12 \left( \begin{matrix}
0 & 1 \cr
1 & 0
\end{matrix} \right)
\quad
t_2=\frac 12 \left( \begin{matrix}
0 & -i \cr
i & 0
\end{matrix} \right)
\quad
t_3=\frac 12 \left( \begin{matrix}
1 & 0 \cr
0 & -1
\end{matrix} \right)
\end{gather*}   
Then 
\begin{gather*}
A^{\mu}=\sum_{a=1}^3 A^{\mu}_at_a\\
[A^2,A^3]=i\left(f_{123}A^2_2A^3_3+f_{132}A^2_2A^3_3\right)t_1\\
+i\left(f_{231}A^2_3A^3_1+f_{213}A^2_3A^3_1\right)t_2+i\left(f_{312}A^2_1A^3_2+f_{321}A^2_2A^3_1\right)t_3
\end{gather*}  
showing that Lemma 2.1 holds in this example. The proposed solutions
make use of the following lemma.

\begin{lemma} 
Let the gauge field have the form
\begin{equation}
A^{\mu}_a=s_aE^{\mu}
\end{equation}  
Then $F^{\mu\nu}_a$ has the form 
\begin{equation}
F^{\mu\nu}_a=s_aG^{\mu\nu}
\end{equation}  
and (2.6) and (2.11) reduce to
\begin{equation}
\begin{gathered}
F^{\mu\nu}_a=\partial^{\mu}A_a^{\nu}-\partial^{\nu}A_a^{\mu}\\
F^a_{\mu\nu}=\partial_{\mu}A^a_{\nu}-\partial_{\nu}A^a_{\mu}
\end{gathered}
\end{equation}  
\end{lemma}

\begin{proof}
Because of (2.11) it suffices to show (2.16) and the first equation in
(2.17). From Lemma 2.1 
\begin{equation}
f_{abc}A^{\mu}_bA^{\nu}_c=\sum_{c>b}f_{abc}\left(s_bE^{\mu}s_cE^{\nu}\\
-s_cE^{\mu}s_bE^{\nu}\right)=0
\end{equation}   
since $s_a$ and $E^{\mu}$ are scalars and commutate.
\end{proof}

The Euler-Lagrange equations for 
$\mathcal{L}=\mathcal{L}(A^{\mu},\partial^{\nu}A^{\mu})$ are
\begin{equation}
\partial^{\nu}\left( \frac {\mathcal{L}}{\partial\left(\partial^{\nu}A_a^{\mu}\right)}\right)=\frac {\partial \mathcal{L}}{\partial A_a^{\mu}}
\end{equation}   
 
\begin{lemma} 
Let 
\begin{equation}
\mathcal{L}=-\frac 14 F^{\mu\nu}_aF_{\mu\nu}^a
\end{equation}   
and $A^{\mu}_a$ be real functions.
Then 
\begin{equation}
\begin{gathered}
\begin{aligned}
&\frac {\partial F^{\mu\nu}_d}{\partial A_a^{\mu}}=-gf_{dac}A_c^{\nu}\\
&\frac {\partial F^d_{\mu\nu}}{\partial A_a^{\mu}}=-gf_{dac}A_c^{\nu}g_{\mu\mu}g_{\nu\nu}\\
&\frac {\partial \mathcal{L}}{\partial A_a^{\mu}}=\frac 12 gf_{abc}A_b^{\nu}F^c_{\mu\nu}\\
&\frac {\partial F^{\mu\nu}_d}{\partial \left(\partial^{\nu}A_a^{\mu}\right)}=-\delta_{ad}\\
&\frac {\partial F^d_{\mu\nu}}{\partial \left(\partial^{\nu}A_a^{\mu}\right)}=\delta_{ad}(g_{\mu\nu}g_{\nu\mu}-g_{\mu\mu}g_{\nu\nu})\\
&\frac {\partial \mathcal{L}}{\partial \left(\partial^{\nu}A_a^{\mu}\right)}=\frac 12 F_{\mu\nu}^a
\end{aligned}
\end{gathered}
\end{equation}  
and the Euler-Lagrange equations are
\begin{equation}
\partial^{\mu}F_{\mu\nu}^a-gf_{abc}A^{\mu}_bF_{\mu\nu}^c=0
\end{equation}  
\end{lemma}

\begin{proof}
Directly computing
\begin{equation}
\begin{gathered}
\frac {\partial F^{\mu\nu}_d}{\partial A_a^{\mu}}=\frac {\partial}{\partial A_a^{\mu}}\left( -gf_{dbc}A^{\mu}_bA^{\nu}_c\right)
=-g\delta_{ab}f_{dbc}A_c^{\nu}= -gf_{dac}A_c^{\nu}
\end{gathered}
\end{equation}  

\begin{equation}
\begin{gathered}
\frac {\partial F^d_{\mu\nu}}{\partial A_a^{\mu}}=\frac {\partial}{\partial A_a^{\mu}}\left( -gf_{dbc}A^b_{\mu}A^c_{\nu}\right)
=\frac {\partial}{\partial A_a^{\mu}}\left( -gf_{dbc}g_{\mu\alpha}A^{\alpha}_bA^c_{\nu}\right)\\
=-gf_{dac}g_{\mu\alpha}\delta_{\alpha\mu}\delta_{ab}A^c_{\nu}= -gf_{dac}A^c_{\nu}g_{\mu\mu}\\
=-gf_{dac}g_{\nu\alpha}A^{\alpha}_cg_{\mu\mu}= -gf_{dac}A_c^{\nu}g_{\mu\mu}g_{\nu\nu}
\end{gathered}
\end{equation}

\begin{equation}
\frac {\partial \mathcal{L}}{\partial A_a^{\mu}}=-\frac 14 \left(\left(\\
\frac {\partial F^{\mu\nu}_d}{\partial A_a^{\mu}}\right)F^d_{\mu\nu}\\
+F^{\mu\nu}_d\left(\frac {\partial F^d_{\mu\nu}}{\partial A_a^{\mu}}\right)\right)
\end{equation}  

\begin{gather*}
=\frac 14 gf_{dac}\left(A_c^{\nu}F^d_{\mu\nu}+F^{\mu\nu}_d A^{\nu}_cg_{\mu\mu}g_{\nu\nu}\right)\\
=\frac 14 gf_{dac}\left(A_c^{\nu}F^d_{\mu\nu}+A^{\nu}_cg_{\mu\alpha}g_{\nu\beta}F^{\alpha\beta}_d \right)\\
=\frac 14 gf_{dac}\left(A_c^{\nu}F^d_{\mu\nu}+A^{\nu}_cF^d_{\mu\nu} \right)\\
=\frac 12 gf_{dac}A_c^{\nu}F^d_{\mu\nu}=\frac 12 gf_{acd}A_c^{\nu}F^d_{\mu\nu}\\
=\frac 12 gf_{abc}A_b^{\nu}F^c_{\mu\nu}
\end{gather*}

\begin{gather*}
\frac {\partial F^{\mu\nu}_d}{\partial \left(\partial^{\nu}A_a^{\mu}\right)}=\frac {\partial}{\partial \left(\partial^{\nu}A_a^{\mu}\right)}\left(\partial^{\mu}A^{\nu}_d-\partial^{\nu}A^{\mu}_d\right)=-\delta_{ad}
\end{gather*}

\begin{gather*}
\frac {\partial F^d_{\mu\nu}}{\partial \left(\partial^{\nu}A_a^{\mu}\right)}
=\frac {\partial}{\partial \left(\partial^{\nu}A_a^{\mu}\right)}\left(\partial_{\mu}A^d_{\nu}-\partial_{\nu}A^d_{\mu}\right)\\
=\frac {\partial}{\partial \left(\partial^{\nu}A_a^{\mu}\right)}
\left(g_{\mu\alpha}g_{\nu\beta}\left(\partial^{\alpha}A^{\beta}_d-\partial^{\beta}A^{\alpha}_d\right)\right)\\
=g_{\mu\alpha}g_{\nu\beta}\delta_{ad}\delta_{\mu\beta}\delta_{\nu\alpha}
-g_{\mu\alpha}g_{\nu\beta}\delta_{ad}\delta_{\nu\beta}\delta_{\mu\alpha}\\
=-\delta_{ad}\left( g_{\mu\nu}g_{\nu\mu}-g_{\mu\mu}g_{\nu\nu}\right)
\end{gather*}

\begin{gather*}
\frac {\partial \mathcal{L}}{\partial \left(\partial^{\nu}A_a^{\mu}\right)}
=-\frac 14\left(-\delta_{ad}F_{\mu\nu}^d-g_{\mu\mu}g_{\nu\nu}\delta_{ad}F_d^{\mu\nu}\right)\\
=\frac 14\left(F_{\mu\nu}^a+g_{\mu\alpha}g_{\nu\beta}F_a^{\alpha\beta}\right)\\
=\frac 14\left(F_{\mu\nu}^a+g_{\mu\alpha}g_{\nu\beta}F_a^{\alpha\beta}\right)\\
=\frac 14\left(F_{\mu\nu}^a+F^a_{\mu\nu}\right)=\frac 12F_{\mu\nu}^a
\end{gather*}

Inserting (2.23) and (2.24) to the Euler-Langange equations (2.19) gives
\begin{gather*}
\partial^{\nu}F_{\mu\nu}^a-gf_{abc}A^{\nu}_bF_{\mu\nu}^c=0
\end{gather*}  
As $F_{\mu\nu}^c=-F_{\nu\mu}^c$ we can also write
\begin{gather*}
\partial^{\nu}F_{\nu\mu}^a-gf_{abc}A^{\nu}_bF_{\nu\mu}^c=0
\end{gather*}  
and changing $\nu$ and $\mu$ yields (2.22)
\begin{gather*}
\partial^{\mu}F_{\mu\nu}^a-gf_{abc}A^{\mu}_bF_{\mu\nu}^c=0
\end{gather*}  

\end{proof}

The Lagrangian can be expressed as 
\begin{equation}
\begin{gathered}
\mathcal{L}=-\frac 14 F^{\mu\nu}_aF^a_{\mu\nu}=-\frac 12 F^{\mu\nu}F_{\mu\nu}|_{\nu>\mu}\\
=-\frac 12 \left(F^{01}_aF_{01}^a+F^{02}_aF_{02}^a+F^{03}_aF_{03}^a
+F^{12}_aF_{12}^a+F^{13}_aF_{13}^a+F^{23}_aF_{23}^a\right)\\
\end{gathered}
\end{equation}   
For Minkowski's metric 
\begin{equation}
\begin{gathered}
\mathcal{L}=-\frac 12 \left(-(F_{01}^a)^2-(F_{02}^a)^2-(F_{03}^a)^2
+(F_{12}^a)^2+(F_{13}^a)^2+(F_{23}^a)^2\right)
\end{gathered}
\end{equation}
since $F^a_{0j}=-F^{0j}_a$ and $F^a_{kj}=F^{kj}_a$ for $1\le j,k\le 3$. 

The Hamiltonian density of a scalar field $\varphi$ is defined as
\begin{equation}
\mathcal{H}=\frac {\partial \mathcal{L}}{\partial \left( \partial_0 \varphi\right)}\left(\partial_0 \varphi \right)-\mathcal{L}=\pi\partial_0\varphi-\mathcal{L}
\end{equation}
where
\begin{gather*}
\pi=\frac {\partial \mathcal{L}}{\partial \left( \partial_0 \varphi\right)}
\end{gather*}

The energy of the field is a conserved property
\begin{equation}
P_0=\int d^3x \mathcal{H}
\end{equation}
In the case of a gauge field $A^{\mu}_a$ we define the Hamiltonian density
as

\begin{equation}
\mathcal{H}=\frac {\partial \mathcal{L}}{\partial \left( \partial_0 A^{\mu}_a\right)}\left(\partial^0 A^{\mu}_a \right)-\mathcal{L}
\end{equation}
where summation over $a$ and $\mu$ is implied. For the Yang-Mills Lagrangian 
we have calculated in Lemma 2.2
\begin{equation}
\frac {\partial \mathcal{L}}{\partial \left( \partial_0 A^{\mu}_a\right)}\left(\partial^0 A^{\mu}_a \right)=\frac 12 F_{\mu 0}^a
\end{equation}
Thus
\begin{equation}
\mathcal{H}=\frac 12 F_{\mu 0}^a\partial^0 A^{\mu}_a -\mathcal{L}
\end{equation}
The energy of the field is a conserved property also in this case
\begin{gather*}
P^0=\int d^3x \mathcal{H}
\end{gather*}

As Minkowski's metric is indefinite, it is sometimes better to move to
either positive or negative definite metric. A convinient choice for computations is the following metric
\begin{equation}
\begin{gathered}
(g_{\mu\nu})_{\mu,\nu}=\left( \begin{matrix}
-1 & 0 & 0 & 0\cr
0 & -1 & 0 & 0\cr
0 & 0 & -1 & 0\cr
0 & 0 & 0 & -1
\end{matrix} \right)
\end{gathered}
\end{equation}   
We will call it negative definite Euclidean metric, though in $\mathbb{R}^4$ 
a proper metric should be positive definite. This negative definite metric
has the advantage that if we do not raise or lower the indices for the $x_0$
coordinate, all formulas remain valid. When we do lower $x_0$ indices, as in
(2.27), there is a change of sign. Additionally, the $x_0$ coordinate must be 
replaced by $ix_0$. This creates an additional $i$ when derivating with
respect to $x_0$.  

The problem setting of CMI uses the more modern algebraic geometric 
formulation where the Yang-Mills action is
\begin{equation}
\mathcal{S}=\frac {1}{4g^2}\int Tr F\wedge *F
\end{equation}
Actually [1] calls this action the Lagrangian but the Langrangian is the
property that is integrated over the space in action. This terminology is 
corrected in [3]. The Yang-Mills equations (2.22) can be expressed with the
Hodge star operator as
\begin{equation}
0=d_A F=d_a *F \quad F=dA+A\wedge A
\end{equation} 
where $d_A$ is the gauge-covariant extension of the exterior derivative.
This is described in a clearer way in [2]. The gauge field $A$ is a one-form
\begin{equation}
A(x)=A_{\mu}^a(x)t^adx^{\mu}
\end{equation}
with the values on the Lie algebra of a compact simple Lie group $G$. 
The curvature is a two-form
\begin{equation}
\begin{gathered}
F=dA+A\wedge A \\
F=F_{\mu\nu}^at^adx^{\mu}\wedge dx^{\nu}\\
F=\partial_{\mu}A_{\nu}-\partial_{\nu}A_{\mu}+f^{abc}A_{\mu}^bA_{\nu}^c
\end{gathered}
\end{equation} 
Instead of the Lagrangian (2.1) we define a four-form
\begin{equation}
\mathcal{A}=Tr F\wedge *F=F_{\mu\nu}^aF_{\mu\nu}^a d^4x
\end{equation} 
and the action is
\begin{equation}
\mathcal{S}=\frac {1}{4g^2}\int \mathcal{A}
\end{equation}
There are differences in the normalization $-\frac 12$ in (2.1) and in
the placement of the coupling constant $g$ in (2.11) and (2.22).
There is also a more essential difference in $\mathcal{A}$ compared to (2.1).
The summation is $F_{\mu\nu}^aF_{\mu\nu}^a$ and not 
$F^{\mu\nu}_aF_{\mu\nu}^a$ as in (2.1). This causes a difference in (2.27)
and it seems that CMI has wanted to pose the problem in Euclidean metric
instead of Minkowski's metric. This is not essential, we get the same result,
apart from a multiplicative constant, for both of the metrics (2.9) 
and (2.33).

\section{Lemmas and theorems}

\begin{lemma} 
Let the gauge field satisfy $A^3_a=0$ for every $a$. 
The Euler-Lagrange equations can be expressed as 
($0\le l, k \le 2$)
\begin{equation}
\partial^3 A^a_l=F_{l3}^a
\end{equation}
\begin{equation}
\partial^3\partial^3 A^a_k=\partial^l F^a_{lk}-gf_{abc}A^l_bF_{l3}^a
\end{equation}
\begin{equation}
\partial^3\partial^l A^a_l-gf_{abc}A^l_bF_{l3}^a=0
\end{equation}
\end{lemma}

\begin{proof}
Let $l\in \{0,1,2\}$. Rewriting (2.22) and inserting the gauge $A^3_b=0$ yields
\begin{equation}
\partial^{3}F_{3\nu}^a+\partial^{l}F_{l\nu}^a-gf_{abc}A^{l}_bF_{l\nu}^c=0
\end{equation} 
As $F^a_{33}=0$ by (2.7) the case $\nu=3$ yields
\begin{equation}
\partial^{l}F_{l 3}^a-gf_{abc}A^{l}_bF_{l 3}^c=0
\end{equation} 
Inserting $A_3^a=0$ to (2.11) yields
\begin{equation}
F_{3 l}^a=\partial_3 A_{l}^a
\end{equation} 
and inserting $\partial^3=-\partial_3$ and $F_{3 l}^a=-F_{l 3}^a$
gives (3.2). Inserting (3.1) to (3.5) yields (3.3). The other values 
$k\in \{0,1,2\}$ in (3.4) give
\begin{equation}
-\partial^{3}F_{3k}^a=\partial^{l}F_{lk}^a-gf_{abc}A^{l}_bF_{lk}^c
\end{equation} 
and inserting (3.6) yields
\begin{gather*}
-\partial^{3}\partial_3 A_k^a=\partial^{l}F_{lk}^a-gf_{abc}A^{l}_bF_{lk}^c
\end{gather*} 
Changing $\partial_3=-\partial^3$ gives (3.2).
\end{proof}

\begin{lemma}
Let the gauge field $A^{\mu}_a$ be of the form
\begin{equation}
A^{\mu}_a=s_aE^{\mu}\quad,\quad E^3=0
\end{equation}
Then the Euler-Lagrange equations are
\begin{equation}
\begin{gathered}
\begin{aligned}
&\partial^3 A^a_l=F_{l3}^a \quad 0\le l, k \le 2\\
&\partial^3\partial^3 A^a_k=\partial^l F^a_{lk}\\
&\partial^3\partial^l A^a_l=0
\end{aligned}
\end{gathered}
\end{equation}
\end{lemma}

\begin{proof}
Since $E^3=0$ Lemma 3.1 applies. By Lemma 2.2 $F_{\mu\nu}^a$ is of the form
\begin{equation}
F_{\mu\nu}^a=s_aG_{\mu\nu}
\end{equation}
As in Lemma 2.2 
\begin{gather*}
f_{abc}A^{\mu}_bF_{\mu\nu}^c=\sum_{c>b}f_{abc}\left(s_bE^{\mu}s_cG_{\mu\nu}-s_cE^{\mu}s_bG_{\mu\nu}\right)=0
\end{gather*}   
since $s_a$ and $E^{\mu}$ are scalars and commutate.
\end{proof}

\begin{lemma}
Let the gauge field $A^{\mu}_{a,m}$ be of the form
\begin{gather*}
A^{\mu}_{a,m}=s_aE^{\mu}_m\quad,\quad E^3_m=0
\end{gather*}
for some finite set of indices $m\in B$
and let us assume that each $A^{\mu}_{a,m}$ is a gauge fiend 
such that 
$A^{\mu}_{a,m}$ and the corresponding $F_{\mu\nu}^{a,m}$ satisfy the 
Euler-Lagrange equations (2.22). Then
\begin{equation}
A^{\mu}_{a}=\sum_m A^{\mu}_{a,m}
\end{equation}
defines the curvature
\begin{equation}
F_{\mu\nu}^{a}=\sum_m F_{\mu\nu}^{a,m}
\end{equation}
such that 
$A^{\mu}_{a}$ and $F_{\mu\nu}^{a}$ satisfy the Euler-Lagrange equations 
(2.22). 
\end{lemma}

\begin{proof}
In this case (2.22) reduces to the linear equations (3.9).
Thus, the sums (3.11), (3.12) also satisfy 
(3.9). The equations (2.17) show that $F_{\mu\nu}^a$ is the sum (3.12).   
\end{proof}

\begin{lemma}
Let the gauge field $A^{\mu}_{a}$ of the form
\begin{gather*}
A^{\mu}_{a}=s_aE^{\mu}\quad, \quad E^3=0
\end{gather*}
be a complex gauge field satisfying (2.22). Let the real and imaginary parts
be
\begin{equation}
A^{\mu}_{a,R}=Re A_a^{\mu}
\end{equation}
and
\begin{gather*}
A^{\mu}_{a,I}=Im A_a^{\mu}
\end{gather*}
and the corresponding curvatures be
\begin{equation}
F_{\mu\nu}^{a,R}=Re F^a_{\mu\nu}
\end{equation}
and
\begin{gather*}
F_{\mu\nu}^{a,I}=Im F^a_{\mu\nu}
\end{gather*}
are real functions satisfying (3.9).
\end{lemma}

\begin{proof}
The equations (2.22) reduce to (3.9) in this case. The equations (3.9) are
linear and the coordinates $x_{\mu}, x^{\mu}$ and partial derivatives 
$\partial^{\mu}, \partial_{\mu}$ are all real. Thus the real and imaginary 
parts of $A^{\mu}_a$ and $F_{\mu\nu}^a$ satisfy (3.9) separately. 
\end{proof}

\begin{lemma}
Let $\alpha_{ij}\in \mathbb{C}$, $0\le i\le 3$, $j=1,2,\dots$, and 
$d_k\not=0$, $0\le k\le 2$, satisfy for every $j$ 
\begin{equation}
\alpha_{3j}^2=\sum_{l=0}^2\alpha_{lj}^2
\end{equation}
\begin{equation}
\sum_{l=0}^2d_l\alpha_{lj}=0
\end{equation}
The condition
\begin{equation}
\alpha_{3j}\alpha_{3k}=\sum_{l=0}^2\alpha_{lj}\alpha_{lk}
\end{equation}
for any $k,j$ with $k>j$ holds if either
\begin{equation}
\sum_{l=0}^2 d_l^2=0 
\end{equation}
or there exists a constant $c$ that for every $j$ 
\begin{equation}
\frac {\alpha_{1j}}{\alpha_{2j}}=c
\end{equation}
The inverse is also true: if (3.17) holds then either (3.18) or (3.19) must 
hold.
\end{lemma}

\begin{proof}
If every $d_l=0$ then (3.18) holds, thus we assume that at least one 
$d_l\not=0$. By symmetry we may assume $d_0\not=0$.
Squaring (3.17) and inserting (3.15) yields 
\begin{gather*}
\sum_{l=0}^2\sum_{m=0}^2\alpha_{lj}^2\alpha_{mk}^2=\sum_{l=0}^2\sum_{m=0}^2\alpha_{lj}\alpha_{lk}\alpha_{mj}\alpha_{mk}
\end{gather*}
Separating $\alpha_{0j}$ terms gives
\begin{equation}
\begin{gathered}
\alpha_{0j}^2(\alpha_{1k}^2+\alpha_{2k}^2)+\alpha_{0k}^2(\alpha_{1k}^2+\alpha_{2k}^2)\\
-2(\alpha_{1j}\alpha_{1k}+\alpha_{2j}\alpha_{2k})\alpha_{0j}\alpha_{0k}
+(\alpha_{ij}\alpha_{2k}-\alpha_{1k}\alpha_{2j})^2=0
\end{gathered}
\end{equation}
Inserting (3.16) in the form
\begin{gather*}
\alpha_{0m}=-\frac {d_1}{d_0}\alpha_{1m}-\frac {d_2}{d_0}\alpha_{2m}
\end{gather*}
for $m\in\{j,k\}$ into (3.20) gives after some calculation
\begin{equation}
\left(\sum_{l=0}^2 d_l^2\right)\left(\alpha_{1j}\alpha_{2k}-\alpha_{1k}\alpha_{2j}\right)^2=0
\end{equation}
proving the lemma.
\end{proof}

\begin{lemma}
Let $d_l,\alpha_{lj}\in \mathbb{C}$, $1\le j\le 3$, $0\le l\le 2$, satisfy
\begin{gather*}
\sum_{l=0}^2d_l\alpha_{lj}=0
\end{gather*}
for every $j$. The vectors
\begin{equation}
\rho_j=\sum_{l=0}^2\alpha_{lj}x_l
\end{equation}
are linearly dependent.
\end{lemma}

\begin{proof}
The determinant of this linear transform
\begin{equation}
\left| \begin{matrix}
\alpha_{01} & \alpha_{11} & \alpha_{21}\cr
\alpha_{02} & \alpha_{12} & \alpha_{22}\cr
\alpha_{03} & \alpha_{13} & \alpha_{23}
\end{matrix} \right|
\end{equation}   
gives zero when the condition
\begin{equation}
d_2\alpha_{2j}=-d_0\alpha_{0j}-d_1\alpha_{1j}
\end{equation}
is inserted.
\end{proof}

\begin{lemma}
Let $\alpha_{ij}\in \mathbb{C}$, $0\le i\le 3$, $j\ge 0$, $d_i\in \mathbb{C}$,
$0\le i \le 3$, satisfy
\begin{equation}
d_3=0
\end{equation}
\begin{equation}
\sum_{l=0}^2d_l^2=0
\end{equation}
\begin{equation}
\sum_{l=0}^2d_l\alpha_{lj}=0
\end{equation}
for every $j$, and
\begin{equation}
\alpha_{3j}^2=\sum_{l=0}^2\alpha_{lj}^2
\end{equation}
for every $j$.
Let $h:\mathbb{C}\to \mathbb{C}$ be holomorphic in some open set $U$
and
\begin{equation}
r_j=\sum_{\mu=0}^3\alpha_{\mu j}x_{\mu}
\end{equation}
Then the gauge field 
\begin{equation}
A_{\mu}^a=s_ad_{\mu}e^{\sum_j h(r_j)} \quad,\quad d_3=0  
\end{equation}
defines $F_{\mu\nu}^a$ which satisfies the Euler-Lagrange equations (2.22).
\end{lemma}

\begin{proof}
We have expressed $A_{\mu}^a$ in contravariant coordinates $x_{\nu}$ instead of covariant coordinates $x^{\nu}$ since the derivations in (3.9) are all 
$\frac{\partial}{\partial x_{\nu}}$.
From (3.30) follows that $A^3=0$ and the gauge field is of the form
\begin{gather*}
A^{\mu}_a=s_aE^{\mu}
\end{gather*}
By Lemma the Euler-Lagrange equations (2.22) reduce to (3.9). Inserting
(3.30) to (2.17) yields
\begin{equation}
F_{\mu\nu}^a=s_ae^{\sum_j h(r_j)}\sum_j(d_{\nu}\alpha_{\mu j}-d_{\mu}\alpha_{\nu j}) h'(r_j)
\end{equation}
Then
\begin{gather*}   
\partial^l F_{lk}^a=s_ae^{\sum_j h(r_j)} d_k\left(\sum_j\left(\sum_{l=0}^2 \alpha^2_{lj}\right)h''(r_j)+\sum_{l=0}^2\left(\sum_j \alpha_{lj}h'(r_j)\right)^2\right)\\
-s_ae^{\sum_j h(r_j)}\sum_j\alpha_{kj} \sum_{l=0}^2 d_l\alpha_{lj} (h''(r_j)\\
-s_ae^{\sum_j h(r_j)}\sum_j\alpha_{kj} h'(r_j)\sum_m h'(r_m)\sum_{l=0}^2 d_l\alpha_{lm} 
\end{gather*}   
Simplifying the expression by (3.27) and (3.28)
\begin{gather*}   
\partial^l F_{lk}^a=s_ae^{\sum_j h(r_j)} d_k\left(\sum_j \alpha^2_{3j}h''(r_j)+\sum_{l=0}^2\left(\sum_j \alpha_{lj}h'(r_j)\right)^2\right)
\end{gather*}   
Using Lemma 3.5 we can express
\begin{gather*}
\sum_{l=0}^2\left(\sum_j \alpha_{lj}h'(r_j)\right)^2=\sum_{l=0}^2 \sum_j \alpha_{lj}^2\left(h'(r_j)\right)^2+2\sum_{l=0}^2\sum_j\sum_{k>j}\alpha_{lj}\alpha_{lk}h'(r_j)h'(r_k)\\
=\sum_j \alpha_{3j}^2\left(h'(r_j)\right)^2+2\sum_j\sum_{k>j}\alpha_{3j}\alpha_{3k}h'(r_j)h'(r_k)
\end{gather*}
Thus
\begin{gather*}
\partial^l F_{lk}^a=s_ae^{\sum_j h(r_j)} d_k\left(\sum_j \alpha^2_{3j}h''(r_j)+\left(\sum_j \alpha_{3j}h'(r_j)\right)^2\right)\\
=\frac {\partial^2}{\partial x^2_3}A_k^a=\partial^3\partial^3 A_k^a
\end{gather*}
The first condition in (3.9) is obvious from (2.11) since $A^a_3=0$, and the last condition in (3.9) hods since by (3.27)
\begin{gather*}
\partial^3\partial^l A_l^a=\partial^3\left(s_a\sum_j\left(\sum_{l=0}^2 d_l\alpha_{lj} \right)h'(r_j)e^{\sum_j h(r_j)} \right)=0.
\end{gather*}
\end{proof}

Let us select three linearly independent vectors 
$r_j=\sum_{\mu=0}^3 \alpha_{\mu j}x_{\mu}$ 
and set the numbers 
$d_{\mu}$ as
\begin{equation}
d_0=\sqrt{2}(1-i)\quad d_1=d_2=1+i \quad d_3=0
\end{equation}  

\begin{equation}
\begin{gathered}
\begin{aligned}
&r_1=x_1-x_2+\sqrt{2}x_3\\
&r_2=x_1-x_2-\sqrt{2}x_3\\
&r_3=i\frac{1}{\sqrt{2}}x_0-x_1+\frac{1}{\sqrt{2}}x_3
\end{aligned}
\end{gathered}
\end{equation}  
Then 
\begin{equation}
\begin{gathered}
\begin{aligned}
&x_1=\frac 14 r_1-\frac 14 r_2 -r_3 +i\frac {1}{\sqrt{2}} x_0\\
&x_2=-\frac 14 r_1-\frac 34 r_2 -r_3 +i\frac {1}{\sqrt{2}} x_0\\
&x_3=\frac {1}{2\sqrt{2}}r_1-\frac {1}{2\sqrt{2}}r_2
\end{aligned}
\end{gathered}
\end{equation}  
These numbers fill the conditions (3.25)-(3.28). We cannot get more than
three linearly independent vectors. From 3.6 it follows that there are only 
two linearly independent linear combinations of $\{x_0,x_1,x_2\}$, and the
third vector is obtained from the gauged coordinate $x_3$: the condition
(3.28) allows two values for $\alpha_{3j}$. Let us express $r_j$ and $h(r_j)$ 
as sums of real and imaginary parts.
\begin{equation}
\begin{gathered}
\begin{aligned}
&r_j=\rho_j+i\sigma_j \quad , \quad \rho_j,\sigma_j\in \mathbb{R}\\
&\rho_1=x_1-x_2+\sqrt{2}x_3\\
&\rho_2=x_1-x_2-\sqrt{2}x_3\\
&\rho_3=-x_1+\frac {1}{\sqrt{2}}x_3\\
&\sigma_1=\sigma_2=0\quad \sigma_3=\frac {1}{\sqrt{2}}x_0\\
&h(r_j)=u(\rho_j,\sigma_j)+iv(\rho_j,\sigma_j)
\end{aligned}
\end{gathered}
\end{equation}   
Thus, if $x_0=0$ then $\sigma_j=0$ for every $j$.
As $h$ is holomorphic, $u$ and $v$ are harmonic functions on $\mathbb{R}^2$.
Thus, $u$ and $v$ cannot be bounded on the whole $\mathbb{R}^2$, but they
can be bounded on a strip $|x_0|\le M$ for a finite $M$. 
Assuming that $h(r_j)$ goes sufficiently fast to zero if $|\rho_j|$ grows,
then for any fixed value of $x_0$ the integral of the Euclidean norm of the 
gauge potential (3.30) over the space coordinates $x_1,x_2,x_3$ is finite.
Also the path integral from finite time $t'$ to another finite 
time $t''$ is finite.
We have much freedom in selecting $u(\rho,0)$. We can choose a real 
analytic function $f:\mathbb{R}\to \mathbb{R}$
that vanishes when $|\rho|$ grows, set 
$u(\rho,0)=f(\rho)$ and extend $u$ to a holomorphic function $h$.
We should expect the solution to behave in the way (3.35) describes. 
It is a localized gauge field, gauge boson,
which moves in the $x_1,x_2$ direction with the speed of light as a function
of $x_0$.
We select a concrete case that gives easy calculations. Let 
\begin{equation}
f(\rho_j)=-\beta^2 \rho_j^2
\end{equation}
and extend it to
\begin{equation}
h(r_j)=-\beta^2 r_j^2
\end{equation}
The real and imaginary parts of $d_{\mu}=c_{\mu}+ie_{\mu}$ are
\begin{equation}
\begin{gathered}
c_0=\sqrt{2}\quad c_1=c_2=1 \quad c_3=0\\
e_0=-\sqrt{2}\quad e_1=e_2=1 \quad e_3=0
\end{gathered}
\end{equation}  
We evaluate the gauge potential at $x_0=0$ and take the real part.
 
\begin{lemma}
Let the gauge field be
\begin{equation}
A^a_{\mu}=s_ad_{\mu}e^{-\beta^2\sum_{j=1}^3 r_j^2}
\end{equation}
where $r_j$ and $d_{\mu}$ are as in (3.32)-(3.33) 
and $\beta,s_a\in \mathbb{R}$.
Then
\begin{equation}
A^{a,R}_{\mu}(0,x_1,x_2,x_3)=Re A^{a}_{\mu}(0,x_1,x_2,x_3)
=s_ac_{\mu}e^{-\beta^2\sum_{j=1}^3 \rho_j^2}
\end{equation}
\begin{equation}
\begin{gathered}
F_{\mu\nu}^{a,R}(0,x_1,x_2,x_3)=Re F_{\mu\nu}^{a}(0,x_1,x_2,x_3)\\
=-s_a2\beta^2e^{-\beta^2\sum_j \rho_j^2}\sum_j Re(d_{\nu}\alpha_{\mu j}-d_{\mu}\alpha_{\nu j})\rho_j
\end{gathered}
\end{equation}
\end{lemma}

\begin{proof}
Inserting $x_0=0$ to (3.39) and (3.31) yields the result. 
\end{proof}

We need the Gaussian integrals
\begin{equation}
\begin{gathered}
\begin{aligned}
&\int_{-\infty}^{\infty} e^{-\frac 12 ax^2}=\sqrt{2\pi}a^{-\frac 12}\\
&\int_{-\infty}^{\infty} xe^{-\frac 12 ax^2}=0\\
&\int_{-\infty}^{\infty} x^2e^{-\frac 12 ax^2}=\sqrt{2\pi}a^{-\frac 32}
\end{aligned}
\end{gathered}
\end{equation}

\begin{lemma}
Let the gauge field satisfy
\begin{equation}
A^{a,R}_{\mu}(0,x_1,x_2,x_3)=s_ac_{\mu}e^{-\beta^2\sum_{j=1}^3 \rho_j^2}
\end{equation}
where $\rho_j$ and $c_{\mu}$ are as in (3.35),(3.38) and 
$\beta,s_a\in \mathbb{R}$.
Then 
\begin{equation}
\int d^3x (A_k^a(0,x_1,x_2,x_3))^2=s_a^2c_k^2{(\frac {\pi}{2})}^{\frac 32}\frac {1}{\beta^3}
\end{equation}
\end{lemma}

\begin{proof}
We change the variables to $y_1,y_2,y_3$
\begin{equation}
\begin{gathered}
\begin{aligned}
&y_1=\sqrt{3}x_1-\frac {2}{\sqrt{3}}x_2-\frac {1}{\sqrt{6}}x_3\\
&y_2=\sqrt{\frac 23}x_2-\frac {1}{\sqrt{3}}x_3\\
&y_3=2x_3
\end{aligned}
\end{gathered}
\end{equation}
Then
\begin{gather*}
\sum_{j=1}^3 \rho_j^2=y_1^2+y_2^2+y_3^2
\end{gather*}
As $y_2$ and $y_3$ are not functions of $x_1$ we can change the order of 
integration
\begin{gather*}
\int d^3x e^{-\beta^2\sum_{j=1}^3 \rho_j^2}
=\int d^3x e^{-\beta^2\sum_{j=1}^3 y_j^2}\\
=\int d^2x e^{-\beta^2(y_2^2+y_3^2)}\int dx_1 e^{-\frac 12 (\sqrt{2}\beta)^2y_1^2}\\
=\int d^2x e^{-\beta^2(y_2^2+y_3^2)}\frac{1}{\sqrt{3}}\int dy_1 e^{-\frac 12 (\sqrt{2}\beta)^2y_1^2}\\
=\int d^2x e^{-\beta^2(y_2^2+y_3^2)}\frac{1}{\sqrt{3}}\sqrt{2\pi}(\sqrt{2}\beta)^{-1}
\end{gather*}
As $y_3$ is not a function of $x_2$ we can change the order of integration
\begin{gather*}
=\frac{1}{\sqrt{3}}\sqrt{2\pi}(\sqrt{2}\beta)^{-1}\int dx_3 e^{-\beta^2y_3^2}\int dx_2  e^{-\frac 12(\sqrt{2}\beta)^2y_2^2}\\
=\frac{1}{\sqrt{3}}\sqrt{\frac 32}\sqrt{2\pi}(\sqrt{2}\beta)^{-1}\int dx_3 e^{-\beta^2y_3^2}\int dy_2  e^{-\frac 12(\sqrt{2}\beta)^2y_2^2}\\
=\frac{1}{\sqrt{3}}\sqrt{\frac 32}(2\pi)(\sqrt{2}\beta)^{-2}\int dx_3 e^{-\beta^2y_3^2}\\
=\frac{1}{\sqrt{3}}\sqrt{\frac 32}\frac 12 (2\pi)(\sqrt{2}\beta)^{-2}\int dy_3 e^{-\beta^2y_3^2}\\
=\frac{1}{\sqrt{3}}\sqrt{\frac 32}\frac 12 (2\pi)^{\frac 32}(\sqrt{2}\beta)^{-3}
\end{gather*}
Thus
\begin {gather*}
\int d^3x e^{-2\beta^2\sum \rho_j^2}=\left(\frac {\pi}{2}\right)^{\frac 32}\frac {1}{\beta^3}
\end{gather*}
\end{proof}

\begin{lemma}
Let the gauge field satisfy
\begin{equation}
A^{a,R}_{\mu}(0,x_1,x_2,x_3)=s_ac_{\mu}e^{-\beta^2\sum_{j=1}^3 \rho_j^2}
\end{equation}
where $\rho_j$ and $c_{\mu}$ are as in (3.35), (3.38) and $\beta,s_a\in \mathbb{R}$,
and 
\begin{gather*}
\mathcal{L_R}=-\frac 14 F_{a,R}^{\mu\nu}F^{a,R}_{\mu\nu}
\end{gather*}
Then in Minkowski's metric (2.9) at $x_0=0$
\begin{equation}
\mathcal{L_R}=0
\end{equation}
while in the negative definite metric (2.33) at $x_0=0$
\begin{equation}
\mathcal{L_R}=-\frac 12\sum_a(2\beta^2 s_a)^2e^{-2\beta^2\sum_j\rho_j^2}
4(4\rho_1^2+4\rho_2^2+\rho_3^2-4\rho_2\rho_3)
\end{equation}
\end{lemma}

\begin{proof}
In Minkowski's metric $\mathcal{L}$ is given by (2.27). Thus
\begin{equation}
\begin{gathered}
\mathcal{L_R}=-\frac 12 \left(-(F_{01}^{a,R})^2-(F_{02}^{a,R})^2-(F_{03}^{a,R})^2+(F_{12}^{a,R})^2+(F_{13}^{a,R})^2+(F_{23}^{a,R})^2\right)
\end{gathered}
\end{equation}
From (3.41) we see that 
\begin{equation}
\begin{gathered}
F_{\mu\nu}^{a,R}(0,x_1,x_2,x_3)=-s_a2\beta^2e^{-\beta^2\sum_j \rho_j^2}\sum_jRe(d_{\nu}\alpha_{\mu j}-d_{\mu}\alpha_{\nu j})\rho_j
\end{gathered}
\end{equation}
The parameters selected in (3.32)-(3.33) are
\begin{gather*}
\begin{aligned}
&\alpha_{01}= 0 \quad \alpha_{02}= 0 \quad \alpha_{03}=i\frac {1}{\sqrt{2}}\\
&\alpha_{11}= 1 \quad \alpha_{12}= 1 \quad \alpha_{13}=-1\\
&\alpha_{21}=-1 \quad \alpha_{22}=-1 \quad \alpha_{23}=0\\
&\alpha_{31}=\sqrt{2} \quad \alpha_{32}=-\sqrt{2} \quad \alpha_{33}=\frac {1}{\sqrt{2}}
\end{aligned}
\end{gather*}
\begin{gather*}
\begin{aligned}
&c_0=\sqrt{2} \quad c_1=c_2=1 \quad c_3=0 \\ 
&e_0=-\sqrt{2} \quad e_1=e_2=1 \quad e_3=0  
\end{aligned}
\end{gather*}
Let us compute the needed components
\begin{gather*}
\sum_{j=1}^3\alpha_{0j}\rho_j=i\frac {1}{\sqrt{2}}\rho_3\quad
\sum_{j=1}^3\alpha_{1j}\rho_j=\rho_1+\rho_2-\rho_3\\ 
\sum_{j=1}^3\alpha_{2j}\rho_j=-\rho_1-\rho_2\quad
\sum_{j=1}^3\alpha_{3j}\rho_j=\sqrt{2}\rho_1-\sqrt{2}\rho_2+\frac {1}{\sqrt{2}}\rho_3
\end{gather*}
\begin{equation}
\begin{gathered}
\sum_{j=1}^3 Re(d_{1}\alpha_{0 j}-d_{0}\alpha_{1 j})\rho_j
=-\frac {1}{\sqrt{2}}\rho_3-c_0\sum_{j=1}^3\alpha_{1j}\rho_j\\
=-\frac {1}{\sqrt{2}}\rho_3-\sqrt{2}(\rho_1+\rho_2-\rho_3)\\
\end{gathered}
\end{equation}
\begin{equation}
\begin{gathered}
\sum_{j=1}^3 Re(d_{2}\alpha_{0 j}-d_{0}\alpha_{2 j})\rho_j
=-\frac {1}{\sqrt{2}}\rho_3-c_0\sum_{j=1}^3\alpha_{2j}\rho_j\\
=-\frac {1}{\sqrt{2}}\rho_3-\sqrt{2}(-\rho_1-\rho_2)
\end{gathered}
\end{equation}
\begin{equation}
\begin{gathered}
\sum_{j=1}^3 Re(d_{3}\alpha_{0 j}-d_{0}\alpha_{3 j})\rho_j
=-c_0\sum_{j=1}^3\alpha_{3j}\rho_j\\
=-2\rho_1+2\rho_2-\rho_3
\end{gathered}
\end{equation}
\begin{gather*}
\sum_{j=1}^3 Re(d_{2}\alpha_{1 j}-d_{1}\alpha_{2 j})\rho_j
=c_2\sum_{j=1}^3\alpha_{1j}\rho_j-c_1\sum_{j=1}^3\alpha_{2j}\rho_j\\
=2\rho_1+2\rho_2-\rho_3
\end{gather*}
\begin{gather*}
\sum_{j=1}^3 Re(d_{3}\alpha_{1 j}-d_{1}\alpha_{3 j})\rho_j
=c_3\sum_{j=1}^3\alpha_{1j}\rho_j-c_1\sum_{j=1}^3\alpha_{3j}\rho_j\\
=-\sqrt{2}\rho_1+\sqrt{2}\rho_2-\frac {1}{\sqrt{2}}\rho_3
\end{gather*}
\begin{gather*}
\sum_{j=1}^3 Re(d_{3}\alpha_{2 j}-d_{2}\alpha_{3 j})\rho_j
=c_3\sum_{j=1}^3\alpha_{2j}\rho_j-c_2\sum_{j=1}^3\alpha_{3j}\rho_j\\
=-\sqrt{2}\rho_1+\sqrt{2}\rho_2-\frac {1}{\sqrt{2}}\rho_3
\end{gather*}
The sum of the squares with the signs as in (3.49) is
\begin{gather*}
-(-\frac {1}{\sqrt{2}}\rho_3-\sqrt{2}(\rho_1+\rho_2-\rho_3))^2
-(-\frac {1}{\sqrt{2}}\rho_3-\sqrt{2}(-\rho_1-\rho_2))^2\\
-(-2\rho_1+2\rho_2-\rho_3)^2
+(2\rho_1+2\rho_2-\rho_3)^2\\
+(-\sqrt{2}\rho_1+\sqrt{2}\rho_2-\frac {1}{\sqrt{2}}\rho_3)^2
+(-\sqrt{2}\rho_1+\sqrt{2}\rho_2-\frac {1}{\sqrt{2}}\rho_3)^2\\
=0
\end{gather*}
Inserting the sum to (3.49) and calculating (3.48) yields
\begin{gather*}
\mathcal{L_R}=\frac 12\sum_a(2\beta^2 s_a)^2e^{-2\beta^2\sum_j\rho_j^2}0=0
\end{gather*}
In the negative definite metric (2.33) holds $g_{\mu\mu}=-1$ for all $\mu$, so
\begin{equation}
\begin{gathered}
\mathcal{L_R}=-\frac 12 \left((F_{01}^{a,R})^2+(F_{02}^{a,R})^2+(F_{03}^{a,R})^2
+(F_{12}^{a,R})^2+(F_{13}^{a,R})^2+(F_{23}^{a,R})^2\right)
\end{gathered}
\end{equation}
Then the sum of the terms is
\begin{gather*}
\mathcal{L_R}=-\frac 12\sum_a(2\beta^2 s_a)^2e^{-2\beta^2\sum_j\rho_j^2}
4(4\rho_1^2+4\rho_2^2+\rho_3^2-4\rho_2\rho_3)
\end{gather*}
\end{proof}

\begin{lemma}
Let the gauge field satisfy
\begin{equation}
A^{a,R}_{\mu}(0,x_1,x_2,x_3)=s_ac_{\mu}e^{-\beta^2\sum_{j=1}^3 \rho_j^2}
\end{equation}
where $\rho_j$ and $c_{\mu}$ are as in (3.35), (3.38) and $\beta,s_a\in \mathbb{R}$.
Then 
\begin{equation}
\int d^3 \mathcal{L_R}=-\frac {1}{\beta} \frac{{\pi}^{\frac 32}}{16}\sum_a s_a^2 B
\end{equation}
where in Minkowski's metric at $x_0=0$
\begin{gather*}
B=0 
\end{gather*}
In the negative definite metric of (2.33)
\begin{gather*}
B=\frac {13}{3}+\frac 23 +4 
\end{gather*}  
\end{lemma}

\begin{proof}
From (3.35) and (3.45) follows that
\begin{gather*}
\begin{aligned}
&\rho_1=\frac {1}{\sqrt{3}}y_1-\frac {1}{\sqrt{6}}y_2+\frac {1}{\sqrt{2}}y_3\\
&\rho_2=\frac {1}{\sqrt{3}}y_1-\frac {1}{\sqrt{6}}y_2-\frac {1}{\sqrt{2}}y_3\\
&\rho_3=-\frac {1}{\sqrt{3}}y_1-\sqrt{\frac 23}y_2
\end{aligned}
\end{gather*}
For Minskowski's metric 
\begin{gather*}
P(\rho)=0=B_1y_1^2+B_2y_2^2+B_3y_3^2+B_4y_1y_2+B_5y_1y_3+B_6y_2y_3
\end{gather*}
where $B_k=0$ for all $k$.
For the metric in (2.33)
\begin{gather*}
P(\rho)=4\rho_1^2+4\rho_2^2+\rho_3^2-4\rho_2\rho_3\\
=B_1y_1^2+B_2y_2^2+B_3y_3^2+B_4y_1y_2+B_5y_1y_3+B_6y_2y_3
\end{gather*}
where
\begin{gather*}
B_1=\frac {13}{3} \quad B_2=\frac 23 \quad B_3=4\\
B_4=-\frac 43\sqrt{2} \quad B_5=-\frac {4}{\sqrt{6}} \quad B_6=-\frac {4}{\sqrt{3}}
\end{gather*}
We do the integration with generic parameters $B_j$. Then 
\begin{gather*}
\int d^3x e^{-\frac 12(2\beta)^2(\rho_1^2+\rho_2^2+\rho_3^2)}P(\rho)\\
=\int d^3x e^{-\frac 12(2\beta)^2(y_1^2+y_2^2+y_3^2)}(B_1y_1^2+B_2y_2^2+B_3y_3^2+B_4y_1y_2+B_5y_1y_3+B_6y_2y_3)\\
\end{gather*}
As $y_2$ and $y_3$ are not functions of $x_1$ we can change the order of integration and change the integration parameter $x_1$ to $y_1$.
\begin{gather*}
=\int d^2x e^{-\frac 12(2\beta)^2(y_2^2+y_3^2)}\left( B_1\int dx_1 y_1^2e^{-\frac 12(2\beta)^2y_1^2}\right.\\
\left.+(B_2y_2^2+B_3y_3^2+B_6y_2y_3)\int dx_1 e^{-\frac 12(2\beta)^2y_1^2}\right.\\
\left.+(B_4y_2+B_5y_3)\int dx_1 y_1e^{-\frac 12(2\beta)^2y_1^2}\right)
\end{gather*}
\begin{gather*}
=\frac {1}{\sqrt{3}}\int d^2x e^{-\frac 12(2\beta)^2(y_2^2+y_3^2)}\left(B_1\int dy_1 y_1^2e^{-\frac 12(2\beta)^2y_1^2}\right.\\
\left.+(B_2y_2^2+B_3y_3^2+B_6y_2y_3)\int dy_1 e^{-\frac 12(2\beta)^2y_1^2}\right.\\
\left.+(B_4y_2+B_5y_3)\int dy_1 y_1e^{-\frac 12(2\beta)^2y_1^2}\right)\\
\end{gather*}
\begin{gather*}
=\frac {1}{\sqrt{3}}\int d^2x e^{-\frac 12(2\beta)^2(y_2^2+y_3^2)}\left(B_1\sqrt{2\pi}\frac {1}{(2\beta)^3}\right.\\
\left.+(B_2y_2^2+B_3y_3^2+B_6y_2y_3)\sqrt{2\pi}\frac {1}{2\beta}\right)
\end{gather*}
As $y_3$ is not a function of $x_2$ we can change the order of integration and change the integration parameter $x_2$ to $y_2$.
\begin{gather*}
=\frac {1}{\sqrt{3}}\sqrt{\frac 32}\sqrt{2\pi}\int dx_3 e^{-\frac 12(2\beta)^2y_3^2}\left(B_1\frac {1}{(2\beta)^3}\int dy_2 e^{-\frac 12(2\beta)^2y_2^2}\right.\\
\left.+B_2\frac {1}{2\beta}\int dy_2 y_2^2 e^{-\frac 12(2\beta)^2 y_2^2}\right.\\
\left.+B_3y_3^2\frac {1}{2\beta}\int dy_2 e^{-\frac 12(2\beta)^2 y_2^2}\right.\\
\left.+B_6y_3\frac {1}{2\beta}\int dy_2 y_2 e^{-\frac 12(2\beta)^2 y_2^2}\right)
\end{gather*}
\begin{gather*}
=\frac {1}{\sqrt{3}}\sqrt{\frac 32}\sqrt{2\pi}\int dx_3 e^{-\frac 12(2\beta)^2 y_3^2}\left(B_1\frac {1}{(2\beta)^3}\sqrt{2\pi}\frac {1}{2\beta}\right.\\
\left.+B_2\frac {1}{2\beta}\sqrt{2\pi}\frac {1}{(2\beta)^3}\right.
\left.+B_3y_3^2\frac {1}{2\beta}\sqrt{2\pi}\frac {1}{2\beta}\right)
\end{gather*}
\begin{gather*}
=\frac {1}{\sqrt{3}}\sqrt{\frac 32}\frac 12 2\pi\int dy_3 e^{-\frac 12(2\beta)^2 y_3^2}\left(B_1\frac {1}{(2\beta)^4}\right.\\
\left.+B_2\frac {1}{(2\beta)^4}+B_3y_3^2\frac {1}{(2\beta)^2}\right)\\
=\frac {1}{\sqrt{3}}\sqrt{\frac 32}\frac 12 (2\pi)^{\frac 32}(B_1+B_2+B_3)\frac{1}{(2\beta)^5}\\
={\pi}^{\frac 32}\frac {1}{(2\beta)^5}(B_1+B_2+B_3)
\end{gather*}
\end{proof}

\begin{lemma}
Let the gauge field satisfy
\begin{equation}
A^{a,R}_{\mu}(0,x_1,x_2,x_3)=s_ac_{\mu}e^{-\beta^2\sum_{j=1}^3 \rho_j^2}
\end{equation}
where $\rho_j$ and $c_{\mu}$ are as in (3.35), (3.38) and $\beta,s_a\in \mathbb{R}$.
Then at $x_0=0$
\begin{equation}
\int d^3 \mathcal{H_R}=\frac {1}{\beta} \frac{{\pi}^{\frac 32}}{16}\sum_a s_a^2 B
\end{equation}
where in Minkowski's metric 
\begin{gather*}
B=0
\end{gather*}
and in the metric (2.33) we get
\begin{gather*}
B=\frac {13}{3}+\frac 23 +4
\end{gather*} 
\end{lemma}

\begin{proof}
From (2.32) for the real gauge field $A_{\mu}^{a,R}$
\begin{gather*}
\mathcal{H_R}=\frac 12 F_{\mu 0}^{a,R}\partial^0 A^{\mu}_{a,R} -\mathcal{L_R}
\end{gather*}
From Lemma 3.8
\begin{equation}
A^{a,R}_{\mu}=s_ac_{\mu}e^{-\beta^2\sum_{j=1}^3 \rho_j^2}\cos\left(\frac {\beta^2}{\sqrt{2}}x_0\right)+s_ae_{\mu}e^{-\beta^2\sum_{j=1}^3 \rho_j^2}\sin\left(\frac {\beta^2}{\sqrt{2}}x_0\right) 
\end{equation}
Thus
\begin{gather*}
\partial^0 A^{a,R}_{\mu}=s_ae^{-\beta^2\sum_{j=1}^3 \rho_j^2}\frac {\beta^2}{\sqrt{2}}\left(-c_{\mu}\sin\left(\frac {\beta^2}{\sqrt{2}}x_0\right)+e_{\mu}\cos\left(\frac {\beta^2}{\sqrt{2}}x_0\right)\right) 
\end{gather*}
as $\partial^0\sum_j\rho_j^2=0$. At $x_0=0$
\begin{gather*}
\partial^0 A^{a,R}_{\mu}=s_ae_{\mu}\frac {\beta^2}{\sqrt{2}}e^{-\beta^2\sum_{j=1}^3 \rho_j^2}
\end{gather*}
Thus
\begin{equation}
\begin{gathered}
\begin{aligned}
&\partial^0 A^{a,R}_{0}(0,x_1,x_2,x_3)=-s_a\beta^2e^{-\beta^2\sum_{j=1}^3 \rho_j^2}\\
&\partial^0 A^{a,R}_{1}(0,x_1,x_2,x_3)=s_a\frac {\beta^2}{\sqrt{2}}e^{-\beta^2\sum_{j=1}^3 \rho_j^2}\\
&\partial^0 A^{a,R}_{2}(0,x_1,x_2,x_3)=s_a\frac {\beta^2}{\sqrt{2}}e^{-\beta^2\sum_{j=1}^3 \rho_j^2}\\
&\partial^0 A^{a,R}_{3}(0,x_1,x_2,x_3)=0
\end{aligned}
\end{gathered}
\end{equation}
From (3.50)-(3.53) 
\begin{gather*}
\begin{aligned}
&\frac 12 F_{00}^{a,R}(0,x_1,x_2,x_3)=0\\
&\frac 12 F_{10}^{a,R}(0,x_1,x_2,x_3)=-\frac 12 s_a2\beta^2e^{-\beta^2\sum_j \rho_j^2}(\frac {1}{\sqrt{2}}\rho_3+\sqrt{2}(\rho_1+\rho_2-\rho_3))\\
&\frac 12 F_{20}^{a,R}(0,x_1,x_2,x_3)=-\frac 12 s_a2\beta^2e^{-\beta^2\sum_j \rho_j^2}(\frac {1}{\sqrt{2}}\rho_3-\sqrt{2}(\rho_1+\rho_2) )\\
&\frac 12 F_{30}^{a,R}(0,x_1,x_2,x_3)=-\frac 12 s_a2\beta^2e^{-\beta^2\sum_j \rho_j^2}(2\rho_1-2\rho_2-\rho_3 )
\end{aligned}
\end{gather*}
Since 
\begin{gather*}
A^{\mu}_{a,R}=g^{\mu\nu}A^{a,R}_{\nu}
\end{gather*}
and in the metric (2.9) $A_0^a=A_a^0$, $A_j^a=-A_a^j$, $j>0$,  
\begin{gather*}
\begin{aligned}
&\partial^0 A^{0}_{a,R}(0,x_1,x_2,x_3)=\partial^0 A^{a,R}_{0}(0,x_1,x_2,x_3)=-s_a\beta^2e^{-\beta^2\sum_{j=1}^3 \rho_j^2}\\
&\partial^0 A^{1}_{a,R}(0,x_1,x_2,x_3)=-\partial^0 A^{a,R}_{1}(0,x_1,x_2,x_3)=-s_a\frac {\beta^2}{\sqrt{2}}e^{-\beta^2\sum_{j=1}^3 \rho_j^2}\\
&\partial^0 A^{2}_{a,R}(0,x_1,x_2,x_3)=-\partial^0 A^{a,R}_{2}(0,x_1,x_2,x_3)=-s_a\frac {\beta^2}{\sqrt{2}}e^{-\beta^2\sum_{j=1}^3 \rho_j^2}\\
&\partial^0 A^{3}_{a,R}(0,x_1,x_2,x_3)=0
\end{aligned}
\end{gather*}
Thus
\begin{gather*}
\frac 12 F_{\mu 0}^{a,R}\partial^0 A^{\mu}_{a,R}
=\sum_a\sum_{\mu=0}^3 \frac 12 F_{\mu 0}^{a,R}\partial^0 A^{\mu}_{a,R}
=\sum_a\left(\frac 12 F_{10}^{a,R}\partial^0 A^{1}_{a,R}+\frac 12 F_{20}^{a,R}\partial^0 A^{2}_{a,R}\right)\\
=\beta^4\frac {1}{\sqrt{2}}\left(\sum_a s_a^2\right)e^{-2\beta^2\sum_{j=1}^3 \rho_j^2}(\frac {1}{\sqrt{2}}\rho_3+\sqrt{2}(\rho_1+\rho_2-\rho_3))\\
+\beta^4\frac {1}{\sqrt{2}}\left(\sum_a s_a^2\right)e^{-2\beta^2\sum_j \rho_j^2}(\frac {1}{\sqrt{2}}\rho_3-\sqrt{2}(\rho_1+\rho_2) )\\
=\beta^4\frac {1}{\sqrt{2}}\left(\sum_a s_a^2\right)e^{-2\beta^2\sum_{j=1}^3 \rho_j^2}(2\frac {1}{\sqrt{2}}-1)\rho_3
\end{gather*}
Inserting $y_1,y_2,y_3$ from (3.45) allows us to perform the integration
\begin{equation}
\begin{gathered}
\int d^3x \frac 12 F_{\mu 0}^{a,R}\partial^0 A^{\mu}_{a,R}\\
\end{gathered}
\end{equation}
\begin{gather*}
=\int d^3x \beta^4\frac {1}{\sqrt{2}}\left(\sum_a s_a^2\right)e^{-2\beta^2\sum_{j=1}^3 y_j^2}(2\frac {1}{\sqrt{2}}-1)\frac 12 y_3
\end{gather*}
\begin{gather*}
=\beta^4\frac {1}{\sqrt{2}}\left(\sum_a s_a^2\right)(2\frac {1}{\sqrt{2}}-1)\frac 12 \int d^2x y_3 e^{-2\beta^2(y_2^2+y_3^2)}\int dx_1 e^{-2\beta^2y_1^2}
\end{gather*}
\begin{gather*}
=\frac{1}{\sqrt{3}}\beta^4\frac {1}{\sqrt{2}}\left(\sum_a s_a^2\right)(2\frac {1}{\sqrt{2}}-1)\frac 12 \int d^2x y_3 e^{-2\beta^2(y_2^2+y_3^2)}\int dy_1 e^{-2\beta^2y_1^2}
\end{gather*}
\begin{gather*}
=\frac{\sqrt{2\pi}}{2\beta}\frac{1}{\sqrt{3}}\beta^4\frac {1}{\sqrt{2}}\left(\sum_a s_a^2\right)(2\frac {1}{\sqrt{2}}-1)\frac 12 \int d^2x y_3 e^{-2\beta^2(y_2^2+y_3^2)}\\
=\frac{\sqrt{2\pi}}{2\beta}\frac{1}{\sqrt{3}}\beta^4\frac {1}{\sqrt{2}}\left(\sum_a s_a^2\right)(2\frac {1}{\sqrt{2}}-1)\frac 12 \int d^x_3 y_3 e^{-2\beta^2y_3^2}\int dx_2 e^{-2\beta^2y_2^2}\\
=\sqrt{\frac 32}\frac{\sqrt{2\pi}}{2\beta}\frac{1}{\sqrt{3}}\beta^4\frac {1}{\sqrt{2}}\left(\sum_a s_a^2\right)(2\frac {1}{\sqrt{2}}-1)\frac 12 \int dx_3 y_3 e^{-2\beta^2y_3^2}\int dy_2 e^{-2\beta^2y_2^2}
\end{gather*}
\begin{gather*}
=\sqrt{\frac 32} \left(\frac{\sqrt{2\pi}}{2\beta}\right)^2\frac{1}{\sqrt{3}}\beta^4\frac {1}{\sqrt{2}}\left(\sum_a s_a^2\right)(2\frac {1}{\sqrt{2}}-1)\frac 12 \int dx_3 y_3 e^{-2\beta^2y_3^2}\\
=\frac 12 \sqrt{\frac 32} \left(\frac{\sqrt{2\pi}}{2\beta}\right)^2\frac{1}{\sqrt{3}}\beta^4\frac {1}{\sqrt{2}}\left(\sum_a s_a^2\right)(2\frac {1}{\sqrt{2}}-1)\frac 12 \int dy_3 y_3 e^{-2\beta^2y_3^2}=0
\end{gather*}
Thus
\begin{gather*}
\int d^3 \mathcal{H_R}=-\int d^3 \mathcal{L_R}
\end{gather*}
and (3.58) follows from Lemma 3.11. If the metric is as in (2.33) then 
$A^0_a=-A_o^a$ but this term disappears and the integral in (3.61) still yields zero. If in addition to changing the metric there has been a replacement 
$x_0\to ix_0$ as is often done in order to move from Minkowski's metric to 
Euclidean metric, derivation with respect to $x_0$ gives an additional $i$.
This changes the coefficients $c_j$ to $e_i$ in some places but (3.61) still
holds because the integral disappears because it is of first order in 
$\rho_j$, and that is also true for the metric (2.33) and a change 
$x_0\to ix_0$. Thus, for the metric (2.33) we get another parameters than
for Minkowski's metric but the form is the same.
\end{proof}

\begin{theorem}
Let the gauge field be
\begin{equation}
A^{a}_{\mu}=s_ad_{\mu}e^{-\beta^2\sum_{j=1}^3 r_j^2}
\end{equation}
where $r_j$ and $d_{\mu}$ are as in (3.33),(3.32) and 
$\beta,s_a\in \mathbb{R}$.
The real part is
\begin{equation}
A^{a,R}_{\mu}(0,x_1,x_2,x_3)=s_ac_{\mu}e^{-\beta^2\sum_{j=1}^3 \rho_j^2}
\end{equation}
where $\rho_j$ and $c_{\mu}$ are as in (3.35),(3.38).
The following statements hold
\begin{equation}
\int d^3x (A_k^a(0,x_1,x_2,x_3))^2=s_a^2c_k^2{\left(\frac {\pi}{2}\right)}^{\frac 32}\frac {1}{\beta^3}
\end{equation}
\begin{equation}
\begin{gathered}
\int d^3x {A^a_{\mu}}^* A^a_{mu}=\sum_a\sum_{k=0}^2\int d^3x (A_k^a(0,x_1,x_2,x_3))^2
=\sum_a s_a^2 \sum_{k=0}^2 c_k^2{\left(\frac {\pi}{2}\right)}^{\frac 32}\frac {1}{\beta^3}
\end{gathered}
\end{equation}
where $A^*$ denotes the complex conjugate of $A$.
The real part of the gauge field $A^{a,R}_{\mu}$ and the real part of 
the curvature $F^{a,R}_{\mu\nu}$ satisfy the Lagrange-Euler equations 
\begin{equation}
\mathcal{L_R}=-\frac 14 F_{a,R}^{\mu\nu}F^{a,R}_{\mu\nu}
\end{equation}
and the energy is
\begin{equation}
P^{0,R}=\int d^3x \mathcal{H_R} =\frac {1}{\beta} \frac{{\pi}^{\frac 32}}{16}\sum_a s_a^2 B
\end{equation}
where in Minkowski's metric 
\begin{gather*}
B=0
\end{gather*}
and in the metric (2.33) 
\begin{gather*}
B=\frac {13}{3}+\frac 23 +4
\end{gather*} 
\end{theorem}

\begin{proof}
From Lemma 3.8 follows that (3.63) is the real part of (3.62). 
The claim (3.64) is shown in Lemma 3.9. The imaginary part in
$A^{a}_{\mu}$ is a phase $e^{-i\sqrt{2}x_0}$ which cancels in
${A^a_{\mu}}* A^a_{mu}$, thus (3.65) holds.
The real part of the gauge field and the curvature satisfy the 
Euler-Lagrange equations by Lemma 3.4, thus (3.66) holds.
In Lemma 12 we showed that at $x_0=0$ equation (3.67) holds. 
As $P^0$ is a conserved
property, see (2.29), (3.67) holds for all values of $x_0$. 
\end{proof}

\begin{theorem}
Let $A=(A_{mu})_{\mu}$, $A_{\mu}=A^{a}_{\mu}t_a$ be a complex 
gauge field defined by
\begin{equation}
A^{a}_{\mu}=s_ad_{\mu}e^{-\beta^2\sum_{j=1}^3 r_j^2}
\end{equation}
The numbers $r_j$ and $d_{\mu}$ are as in (3.33),(3.32) and 
$\beta, s_a\in \mathbb{R}$.
The norm is
\begin{equation}
||A||=\int d^3x {A^a_{\mu}}^* A^a_{\mu}=\sum_a s_a^2 \sqrt{2}{\pi}^{\frac 32}\frac {1}{\beta^3}
\end{equation}
where $A^*$ denotes the complex conjugate of $A$.
The gauge field and the corresponding curvature satisfy Euler-Lagrange 
equations for
\begin{equation}
\mathcal{L}=-\frac 14 F_{a}^{\mu\nu}F^{a}_{\mu\nu}
\end{equation}
In both metrics (2.9) and (2.33)  
\begin{equation}
E_{\beta}=\frac {P^0}{||A||}=\beta^2 C
\end{equation}
where
\begin{equation}
P^0=\int d^3x \mathcal{H} =\sum_a s_a^2{\pi}^{\frac 32}\sqrt{2}C\frac {1}{\beta}
\end{equation}
and $C$ is a nonegative constant. 
\end{theorem}

\begin{proof}
In the case of a complex field, the Lagrangian has two
parts, the real and the imaginary. If the field defines a solution to the
Euler-Langange equations, the energy (2.29) is conserved. Thus, also the 
imaginary part is conserved though we only computed the real part. We get
the same dependence of $\beta$ for the imaginary part. 
For the real part of the Lagrangian we get from Theorem 3.13
\begin{equation}
\frac {P_{0,R}}{||A||}=\beta^2\frac {1}{16\sqrt{2}} B
\end{equation}
where we have inserted $\sum_{k=0}^2c_k^2=4$. Including the imaginary part
changes the constant, but it is nonnegative. 
\end{proof}

We can find a gauge field that gives positive energy for Minkowski's metric as
as sum.

\begin{lemma}
Let the gauge field be
\begin{equation}
A^{a}_{\mu}=s_ad_{\mu 1}e^{-\beta^2\sum_{j=1}^3 r_{j2}^2}
+s_ad_{\mu 2}e^{-\beta^2\sum_{j=1}^3 r_{j2}^2}
\end{equation}
where  $\beta,s_a\in \mathbb{R}$ and
\begin{gather*}
r_{jk}=\rho_{j,k}+i\sigma_{j,k}=\sum_{\mu=0}^3 \alpha_{\mu,j,k}x_{\mu}\\
d_{lk}=c_{lk}+ie_{lk}
\end{gather*}

\begin{gather*}
\begin{aligned}
&\alpha_{011}= 0 \quad \alpha_{021}= 0 \quad \alpha_{031}=i\frac {1}{\sqrt{2}}\\
&\alpha_{111}= 1 \quad \alpha_{121}= 1 \quad \alpha_{131}=-1\\
&\alpha_{211}=-1 \quad \alpha_{221}=-1 \quad \alpha_{231}=0\\
&\alpha_{311}=\sqrt{2} \quad \alpha_{321}=-\sqrt{2} \quad \alpha_{331}=\frac {1}{\sqrt{2}}
\end{aligned}
\end{gather*}

\begin{gather*}
\begin{aligned}
&c_{01}=\sqrt{2} \quad c_{11}=1 \quad c_{21}=1 \quad c_{31}=0 \\ 
&e_{01}=-\sqrt{2} \quad e_{11}=1 \quad e_{22}=1 \quad e_{31}=0  
\end{aligned}
\end{gather*}

\begin{gather*}
\begin{aligned}
&\alpha_{012}= 0 \quad \alpha_{022}= 0 \quad \alpha_{032}=i\frac {1}{\sqrt{2}}\\
&\alpha_{112}= 1 \quad \alpha_{122}= 1 \quad \alpha_{132}=-1\\
&\alpha_{212}= 1 \quad \alpha_{222}= 1 \quad \alpha_{232}=0\\
&\alpha_{312}=\sqrt{2} \quad \alpha_{322}=-\sqrt{2} \quad \alpha_{332}=\frac {1}{\sqrt{2}}
\end{aligned}
\end{gather*}

\begin{gather*}
\begin{aligned}
&c_{02}=\sqrt{2} \quad c_{12}=1 \quad c_{22}=-1 \quad c_{32}=0 \\ 
&e_{02}=-\sqrt{2} \quad e_{12}=1 \quad e_{22}=-1 \quad e_{32}=0  
\end{aligned}
\end{gather*}
Then
\begin{equation}
A^{a,R}_{\mu}(0,x_1,x_2,x_3)=s_ac_{\mu 1}e^{-\beta^2\sum_{j=1}^3 \rho_{j2}^2}
+s_ac_{\mu2}e^{-\beta^2\sum_{j=1}^3 \rho_{j2}^2}
\end{equation}
and 
\begin{gather*}
\mathcal{L_R}=-\frac 14 F_{a,R}^{\mu\nu}F^{a,R}_{\mu\nu}
\end{gather*}
In Minkowski's metric (2.9) at $x_0=0$
\begin{equation}
\begin{gathered}
\mathcal{L_R}=-2\frac 12\sum_a(2\beta^2 s_a)^2e^{-2\beta^2\sum_jy_j^2}\left(
-\frac {13}{3}y_1^2+8y_2^2-\frac{170}{21}y_3^2+\frac {8}{21}\sqrt{14}y_1y_3\right)
\end{gathered}
\end{equation}
where
\begin{equation}
\begin{gathered}
y_1=\sqrt{6}x_1-\frac{1}{\sqrt{3}}x_3\quad y_2=2x_2\quad y_3=\sqrt{\frac {14}{3}}x_3
\end{gathered}
\end{equation}
\end{lemma}

\begin{proof}
By Lemma 3.3 the sum of solutions satisfying (3.9) is also a solution 
satisfying (3.9).  
Most of the proof is as in Lemma 3.10. We are interested in the cross term 
that comes from squaring $F^{a,R}_{\mu\nu}$. 
As it has two components from the two
fields in the sum, the squares of each field give two squares and 
a cross term (twice the product of the terms).
Both of the squares disappear as in Lemma 3.10
but the cross term gives the term in (3.76) and it does not disappear.
We compute only this term in detail.
Let us notice that $r_{31}=r_{32}$ and for simplicity we will write
\begin{gather*}
\rho_3=\rho_{31}=\rho_{32}=-x_1+\frac {1}{\sqrt{2}}x_3
\end{gather*}
We notice that
\begin{equation}
\sum_{j=1}^3\rho_{j1}^2+\sum_{j=1}^3\rho_{j2}^2=y^2_1+y^2_2+y_3^2
\end{equation}
Let us compute the needed components
\begin{gather*}
\begin{aligned}
&\sum_{j=1}^3\alpha_{0j1}\rho_{j1}=i\frac {1}{\sqrt{2}}\rho_{3}\\
&\sum_{j=1}^3\alpha_{1j1}\rho_{j1}=\rho_{11}+\rho_{21}-\rho_3=2x_1-2x_2-\rho_3\\
&\sum_{j=1}^3\alpha_{2j1}\rho_{j1}=-\rho_{11}-\rho_{21}=-2x_1+2x_2\\
&\sum_{j=1}^3\alpha_{3j1}\rho_{j1}=\sqrt{2}\rho_{11}-\sqrt{2}\rho_{21}+\frac {1}{\sqrt{2}}\rho_3=4x_3+\frac {1}{\sqrt{2}}\rho_3
\end{aligned}
\end{gather*}
\begin{gather*}
\begin{aligned}
&\sum_{j=1}^3\alpha_{0j2}\rho_{j2}=i\frac {1}{\sqrt{2}}\rho_{3}\\
&\sum_{j=1}^3\alpha_{1j2}\rho_{j2}=\rho_{12}+\rho_{22}-\rho_3=2x_1+2x_2-\rho_3\\
&\sum_{j=1}^3\alpha_{2j2}\rho_{j2}=\rho_{12}+\rho_{22}=2x_1+2x_2\\
&\sum_{j=1}^3\alpha_{3j2}\rho_{j2}=\sqrt{2}\rho_{12}-\sqrt{2}\rho_{22}+\frac {1}{\sqrt{2}}\rho_3=4x_3+\frac {1}{\sqrt{2}}\rho_3
\end{aligned}
\end{gather*}
\begin{equation}
\begin{gathered}
\sum_{j=1}^3 Re(d_{11}\alpha_{0j1}-d_{01}\alpha_{1j1})\rho_{j1}
=-\frac {1}{\sqrt{2}}\rho_3-\sqrt{2}(\rho_{11}+\rho_{21}-\rho_3)\\
=\frac {1}{\sqrt{2}}x_3-2\sqrt{2}x_1+2\sqrt{2}x_2
\end{gathered}
\end{equation}
\begin{equation}
\begin{gathered}
\sum_{j=1}^3 Re(d_{12}\alpha_{0j2}-d_{02}\alpha_{1j2})\rho_{j2}
=-\frac {1}{\sqrt{2}}\rho_3-\sqrt{2}(\rho_{12}+\rho_{22}-\rho_3)\\
=\frac {1}{\sqrt{2}}x_3-2\sqrt{2}x_1-2\sqrt{2}x_2
\end{gathered}
\end{equation}
\begin{equation}
\begin{gathered}
\sum_{j=1}^3 Re(d_{21}\alpha_{0j1}-d_{01}\alpha_{2j1})\rho_{j1}
=-\frac {1}{\sqrt{2}}\rho_3-\sqrt{2}(-\rho_{11}-\rho_{21})\\
=-\frac {1}{\sqrt{2}}x_3+2\sqrt{2}x_1-2\sqrt{2}x_2
\end{gathered}
\end{equation}
\begin{equation}
\begin{gathered}
\sum_{j=1}^3 Re(d_{22}\alpha_{0j2}-d_{02}\alpha_{2j2})\rho_{j2}
=\frac {1}{\sqrt{2}}\rho_3-\sqrt{2}(\rho_{11}+\rho_{21})\\
=\frac {1}{\sqrt{2}}x_3-2\sqrt{2}x_1-2\sqrt{2}x_2
\end{gathered}
\end{equation}
\begin{equation}
\begin{gathered}
\sum_{j=1}^3 Re(d_{31}\alpha_{0j1}-d_{01}\alpha_{3j1})\rho_{j1}
=-2\rho_{11}+2\rho_{21}-\rho_3\\
=-4\sqrt{2}x_3-\rho_3
\end{gathered}
\end{equation}
\begin{equation}
\begin{gathered}
\sum_{j=1}^3 Re(d_{32}\alpha_{0j2}-d_{02}\alpha_{3j2})\rho_{j2}
=-2\rho_{12}+2\rho_{22}-\rho_3\\
=-4\sqrt{2}x_3-\rho_3
\end{gathered}
\end{equation}
\begin{gather*}
\sum_{j=1}^3 Re(d_{21}\alpha_{1j1}-d_{11}\alpha_{2j1})\rho_{j1}
=2\rho_{11}+2\rho_{21}-\rho_3\\
=4x_1-4x_2-\rho_3
\end{gather*}
\begin{gather*}
\sum_{j=1}^3 Re(d_{22}\alpha_{1j2}-d_{12}\alpha_{2j2})\rho_{j2}
=-2\rho_{12}-2\rho_{22}+\rho_3\\
=-4x_1-4x_2+\rho_3
\end{gather*}
\begin{gather*}
\sum_{j=1}^3 Re(d_{31}\alpha_{1j1}-d_{11}\alpha_{3j1})\rho_{j1}
=-\sqrt{2}\rho_{11}+\sqrt{2}\rho_{21}-\frac {1}{\sqrt{2}}\rho_3
=-4x_3-\frac {1}{\sqrt{2}}\rho_3
\end{gather*}
\begin{gather*}
\sum_{j=1}^3 Re(d_{32}\alpha_{1j2}-d_{12}\alpha_{3j2})\rho_{j2}
=\sqrt{2}\rho_{12}+\sqrt{2}\rho_{22}+\frac {1}{\sqrt{2}}\rho_3
=4x_3+\frac {1}{\sqrt{2}}\rho_3
\end{gather*}
\begin{gather*}
\sum_{j=1}^3 Re(d_{31}\alpha_{2j1}-d_{21}\alpha_{3j1})\rho_{j1}
=-\sqrt{2}\rho_{11}+\sqrt{2}\rho_{21}-\frac {1}{\sqrt{2}}\rho_3
=-4x_3-\frac {1}{\sqrt{2}}\rho_3
\end{gather*}
\begin{gather*}
\sum_{j=1}^3 Re(d_{32}\alpha_{2j2}-d_{22}\alpha_{3j2})\rho_{j2}
=-\sqrt{2}\rho_{12}+\sqrt{2}\rho_{22}-\frac {1}{\sqrt{2}}\rho_3
=-4x_3-\frac {1}{\sqrt{2}}\rho_3
\end{gather*}
The cross term in Minkowski's metric is
\begin{gather*}
-(\frac {1}{\sqrt{2}}\rho_3-2\sqrt{2}x_1+2\sqrt{2}x_2)(\frac {1}{\sqrt{2}}\rho_3-2\sqrt{2}x_1-2\sqrt{2}x_2)\\
-(-\frac {1}{\sqrt{2}}\rho_3+2\sqrt{2}x_1-2\sqrt{2}x_2)(\frac {1}{\sqrt{2}}\rho_3-2\sqrt{2}x_1-2\sqrt{2}x_2)\\
-(-4\sqrt{2}x_3-\rho_3)^2\\
+(4x_1-4x_2-\rho_3)(-4x_1-4x_2+\rho_3)\\
+(-4x_3-\frac{1}{\sqrt{2}}\rho_3)(4x_3+\frac{1}{\sqrt{2}}\rho_3)\\
+(-4x_3-\frac{1}{\sqrt{2}}\rho_3)^2
\end{gather*}

\begin{gather*}
=-16x_1^2+16x_2^2-32x_3^2-2\rho_3^2+\rho_3(8x_1-8\sqrt{2}x_3)\\
=-26x_1^2+16x_2^2-41x_3^2+14\sqrt{2}x_1x_3\\
=-\frac {13}{3}y_1^2+8y_2^2-\frac{170}{21}y_3^2+\frac {8}{21}\sqrt{14}y_1y_3
\end{gather*}
Inserting this result as in Lemma 3.10 yields the claim. 
\end{proof}

\begin{theorem}
Let $A=(A_{mu})_{\mu}$, $A_{\mu}=A^{a}_{\mu}t_a$ be a complex 
gauge field as in Lemma 3.15.
The gauge field and the corresponding curvature satisfy Euler-Lagrange 
equations for
\begin{equation}
\mathcal{L}=-\frac 14 F_{a}^{\mu\nu}F^{a}_{\mu\nu}
\end{equation}
In Minkowski's metric  
\begin{equation}
E_{\beta}=\frac {P^0}{||A||}=\beta^2 C
\end{equation}
where $C$ is a positive constant. 
\end{theorem}

\begin{proof}
The Lagrangian is computed in Lemma 3.15. As in Lemma 3.9 the norm $||A||$
is not zero and depends on $\beta£$ as $\beta^{-3}$. As in Lemma 3.11 
the Lagrangian in (3.15) when integrated over the
space coordinates is proportional to $B=B_1+B_2+B_3=-\frac {13}{3}+8-\frac{170}{21}$
which is nonzero and the integral over space coordinates does not vanish. As
in Lemma 3.12 the first part of the Hamiltonian density (2.30) does not
contribute to the integral:
\begin{gather*}
\int d^3 \mathcal{H_R}=-\int d^3 \mathcal{L_R}
\end{gather*}
The rest is as in Theorem 3.14.
\end{proof}

\section{Mass Gap and Quantization of Yang-Mills fields}

The first question is what is mass gap. L. Faddeev explains the issue in [2]
but let us proceed in a similar way as in [6] from quantum mechanics and
scalar quantum field theory to quantum Yang-Mills theory.
We take a simple scalar wave function of one variable
\begin{equation}
\varphi(x_1)=e^{-\frac 12 a x_1^2}
\end{equation}
Then 
\begin{equation}
\left(\frac {1}{a^2}\frac {\partial^2}{\partial x_1^2}+\frac{1}{a}\right)\varphi(x_1)=x_1^2\varphi(x_1)
\end{equation}
It follows that 
\begin{equation}
\frac {1}{a}\int_{-\infty}^{\infty}dx_1\varphi(x_1)=\sqrt{2\pi}a^{-\frac 32}
\end{equation}
while also
\begin{equation}
\int_{-\infty}^{\infty}dx_1x_1^2\varphi(x_1)=\sqrt{2\pi}a^{-\frac 32}
\end{equation}
The function $\varphi(x_1)$ is time-independent as it does not
depend on $x_0$. We can consider it as a state in the 
Schr\"odinger picture
\begin{equation}
|q>=|q>_S
\end{equation}
We can consider 
\begin{equation}
\hat{A}=A(x_1)=\frac {1}{a^2}\frac {\partial^2}{\partial x_1^2}+\frac{1}{a}
\end{equation}
as an operator acting on the state $|q>$. In order to
take an inner product of $\hat{A}|q>$ with another state
$|q'>$ corresponing to the field $\varphi'(x_1)$
it is more convenient to define the operator as   
\begin{equation}
\hat{H}=H(x'_1,x_1)=\left(\frac {1}{a^2}\frac {\partial}{\partial x'_1}\frac {\partial}{\partial x_1}+\frac{1}{a}\right)\delta(x_1-x'_1)
\end{equation}
Then 
\begin{equation}
<q'|\hat{H}|q>=\int dx'_1 \int dx_1 \varphi'(x'_1)H(x'_1,x_1)\varphi(x_1)
\end{equation}
Especially
\begin{equation}
\begin{gathered}
<q|\hat{H}|q>=\int dx'_1 \int dx_1 \varphi(x'_1)H(x'_1,x_1)\varphi(x_1)\\
=\int dx_1 x_1^2e^{-2\frac 12 a x_1^2}=\sqrt{2\pi}{(\sqrt{2}a)}^{-\frac 32}
\end{gathered}
\end{equation}
while
\begin{equation}
\begin{gathered}
<q|q>=\int dx'_1 \int dx_1 \varphi(x'_1)\delta(x_1-x'_1)\varphi(x_1)\\
=\int dx_1 \phi(x_1)^*\phi(x_1)
=\int dx_1 e^{-2\frac 12 a x_1^2}=\sqrt{2\pi}{(\sqrt{2}a)}^{-\frac 12}
\end{gathered}
\end{equation}
Thus
\begin{equation}
<q|\hat{H}|q>=E<q|q> \quad E=\frac {1}{\sqrt{2}a}
\end{equation}
Thus, $E$ is the expectation value of the operator $\hat H$ at the 
state $|q>$. Let us assume that the state $|q>$ is expressed as a linear 
combination of the eigenstates of the Hamiltonian operator $\hat H$. If
$E$ can be selected arbitrarily small then we can select a sequence of states 
$|q_n>$ where $E_n$ goes to zero. This means that either the sequence of the
states $|q_n>$ converges to the vacuum state, or that there is no minimal
positive eigenstate for $\hat H$. The state where $|q_n>$ converges
if $a\to\infty$ is zero, which is not a vacuum state. 
We conclude that there is no
minimal positive eighenvalue for $\hat H$, i.e., there is no mass gap. 
We can also write the equation with
the Hamiltonian density $\mathcal{H}$
\begin{equation}
\int dx'_1 \int dx_1 \varphi(x'_1)H(x'_1,x_1)\varphi(x_1)=\int dx_1 \mathcal{H}
\end{equation}

The set of eigenvalues of the Hamiltonian operator forms the energy-mass 
spectrum of the field. The zero function $\varphi(x_1)=0$ always satisfies
the eigenvalue equation but it is not an acceptable eigenstate since it
has zero form. There is assumed to exist an eigenstate with eigenvalue zero, 
the vacuum. The vacuum is not unique in all theories, but it must be unique
in a theory filling Wightman's axioms. 
If there is a minimum positive value $E$ in the 
energy-mass spectrum, we say that there is a mass gap. The eigenstates are
closely related to a parameter called mass because the physical interpretation of the parameter $m$ in an equation
\begin{equation}
\left(\partial_{\mu}\partial^{\mu}+m^2\right)\varphi=0 
\end{equation}
is mass.

Let us now proceed to find the Hamiltonian operator for the Hamiltonian 
density $\mathcal{H}_R$ in Lemma 3.12. We notice that in Lemma 3.12
\begin{gather*}
\int d^3 \mathcal{H_R}=-\int d^3 \mathcal{L_R}
\end{gather*}  
From Lemmas 3.10 and 3.11 we see that the Lagrangian can be expressed in 
variables $y_1,y_2,y_3$ as
\begin{gather*}
\mathcal{L_R}=\frac 12\sum_a(2\beta^2 s_a)^2e^{-2\beta^2\sum_j\rho_j^2}P(\rho)\\
=\frac 12\sum_a(2\beta^2 s_a)^2e^{-2\beta^2\sum_jy_j^2}4(B_1y_1^2+B_2y_2^2+B_3y_3^2+B_4y_1y_2+B_5y_1y_3+B_6y_2y_3)
\end{gather*}
We can ignore the terms $(B_4y_1y_2+B_5y_1y_3+B_6y_2y_3)$ since they disappear
in the integration in Lemma 3.11 and conclude that the Hamiltonian density
in the case of this field takes the form
\begin{equation}
\begin{gathered}
\mathcal{H_R}=Ce^{-\frac 12(2\beta)^2\sum_jy_j^2}(B_1y_1^2+B_2y_2^2+B_3y_3^2)
\quad C=16\sqrt{2}\sum_a s_a^2 \beta^4
\end{gathered} 
\end{equation}
Comparing this expression with (4.4) we can write the Hamiltonian operator
as
\begin{equation}
\hat {H}=H(y',y)=C\sum_{j=1}^3B_j\left(\frac{1}{(2\beta)^4}\frac {\partial}{\partial y'_j}\frac {\partial}{\partial y_j}+\frac {1}{(2\beta)^2}\right)\prod_{j=1}^3 \delta(y_j-y'_j)
\end{equation}
where $y'=(y'_1,y'_2,y'_3)$, $y=(y_1,y_2,y_3)$. 
Let us mention that the Hamiltonian takes this simple form only for the 
field (3.62), not for every field. 
Then
\begin{equation}
<A|\hat{H}|A>=E<A|A>
\end{equation}
takes the form
\begin{equation}
\int d^3y'\int d^3y A(y')H(y',y)A(y)=E\int d^3y'\int d^3y A(y')\delta(y-y')A(y)
\end{equation}
which is the same as
\begin{equation}
\int d^3y \mathcal{H}_R=E\int d^3y A(y)^*A(y)
\end{equation} 
We refer to formulae (2.19), (4.31) and (17.50) in [6] 
for the connection between the Hamiltonian operator and (4.7) and (4.15). 
There are of course many approaches but following the approach in [6] 
the operators (4.7) and (4.15) can be understood to describe the Hamiltonian 
operator for a field theory. 

We see that as $\beta$ in Theorem 3.14 can be freely selected,  
either there is no mass gap or vacuum is not unique, 
provided that the gauge fields $A^a_{\mu}$ in (3.68) and (3.74) are 
acceptable.
We obtained $B=0$ for the energy of the real part in
Minkowski's metric in Lemma 3.12 for the gauge field in (3.68).
While we gave another gauge field with positive energy in (3.74), 
let us notice that the result is negative for the CMI problem also if the 
constant $C=0$. If $C=0$ it implies that the vacuum is not
unique and contradicts Wightman's Axiom II that demands that with the 
exception of vacuum all states have positive energy. 

Let us now continue to the question if $A^a_{\mu}$ in (3.68) and (3.74)
can be excluded 
in a non-trivial quantum field theory for the Yang-Mills Lagrangian (2.1).

Quantization of the Yang-Mills theory can be made by two methods; 
by the path integral method, or by axiomatic quantum field theory.
Canonical quantization is also possible but considered difficult. 
Let us first look at the path 
integral method. Basically quantization of a Yang-Mills field is made by
writing the ground-state-to-ground-state amplitude $W[J]$ as a path integral
\begin{equation}
W[J^{\mu}_a]\sim \int \mathcal{D}A^{\mu} exp\left\{ -i\hbar^{-1} \int d^4x \left(\mathcal{L}_{YM}+J_a^{\mu}A^a_{\mu}\right)\right\}
\end{equation}   
However, there are problems in the path integral and the form (4.19) 
is not to be
followed precisely. The path integral may become infinite for a number of 
reasons and a proper quantization should avoid these pitfalls. The character
of such arguments is either mathematical or physical. For instance, the reason
why the field should disappear when the space coordinates grow is physical.
Mere integrability of a function does not require that it vanishes in 
infinity as positive and negative parts can cancel.  

The discussion in [6] on page 117 mentions the need for fixing the 
gauge in a case where there are infinitely many $A^{\mu}$ related by a gauge
transform, and mentions divergences even in the case that the coupling 
constant $g=0$ in (2.6). 

However, there are more problems in (4.1) when considering $A^a_{\mu}$ in 
(3.68). As can be seen in (3.34), the field $A^a_{\mu}$ is a localized wave
packet, gauge boson, 
that moves with the speed of light ($x_0$) in the $(x_1,x_2)$ plane
to the direction $e_1+e_2$ where $e_j$ is the unit vector of the $j$th 
coordinate. This is very natural behavior for a localized wave packet. It
cannot stay in a limited box, and it cannot be bounded in the time dimension 
$x_0$ because it stays localized and the energy is conserved as $A^a_{\mu}$ is
a solution to the Euler-Lagrange equations. Thus, the groud-state-to-ground
state amplitude
\begin{equation}
W[J]\sim \lim_{{t''\to \infty},{t'\to -\infty}} <q'',t''|q',t'>^J
\end{equation}
is not a proper quantity for this field. We can calculate transitions between
any finite times $t'$ to $t''$, and then the path integral is finite.
\begin{equation}
<q'',t''|q',t'>^J \sim \int \mathcal{D}A^{\mu} exp\left\{ -i\hbar^{-1} \int _{t'}^{t''}dt \int d^3x \left(\mathcal{L}_{YM}+J_a^{\mu}A^a_{\mu}\right)\right\}
\end{equation}   
As essential problem is that as $h:\mathbb{C}\to\mathbb{C}$ in (3.30) must
be holomorphic so that differentiation can be made, its real and imaginary 
parts cannot be bounded. We give a physicality argument  

It is reasonable to require that the field vanishes when the space coordinates
go to $\pm \infty$. However, the time coordinate is different.  
The future cannot effect the past and therefore possible divergences in the 
future are not an appropriate boundary condition for a physical problem 
setting. Likewise, there may well be a finite beginning instance of the time 
and therefore extension of $x_0$ to $-\infty$ is highly speculative.  
Thus, the integration over $x_0$ in $W[J]$ is physically motivated only
between two finite time instances $t'$ and $t''$. Accepting that this argument
for avoiding infinities in the path integral is as reasonable as other tricks
that have been used to the same goal in the semi-mathematical path integral
method, such as cutoffs, renormalization, gauge fixing, etc., the gauge field
$A^a_{\mu}$ in (3.68) is acceptable in a non-trivial quantum Yang-Mills theory
created through the path integral method. 

There are other possible mechanisms to render (4.21) finite. In perturbation 
theory the path integral cannot include solutions to the linear Lagrange 
equations. Thus, (3.68) could be excluded. As there is no other motivation
for exclusion than obtaining a finite integral, one should consider the 
physicality argument above as a more acceptable way to get a finite (4.21). 
In any case, there are various ad hoc methods used in the path integral method
that have the aim of removing infinities from (4.21). 

The gauge fields of the type (3.30) admit a non-trivial quantum field theory
for (2.1). We can divide the path integral to two (and later more) parts, where
the first part only has fields of the type (3.30). They are easy to handle,
sums and real parts also satisfy the Euler-Lagrange equations, as is shown 
in Lemmas 3.3 and 3.4. Sums of these type of fields with different $s_a$
yield equations that involve the structure coefficients $f_{abc}$ and may have
some special solutions. We can briefly look at a sum of solutions of the type
(3.30) with different $s_a$.

\begin{lemma} Let the gauge field 
\begin{equation}
A^{\mu}_a=(s_{a,1}d_{\mu,1}+s_{a,2}d_{\mu,2}) e^{\sum_j h(r_j)} \quad A^3_a=0
\end{equation}
be a solution to (2.22). Then $h$ satisfies an equation of the type
\begin{equation}
\sum_jC_{akj}h''(r_j)+\sum_{j,m}D_{akjm}h'(r_j)h'(r_m)+\sum_jE_{akj}h'(r_j)e^{\sum_j h(r_j)}+F_{ak}e^{\sum 2h(r_j)}
\end{equation}
where $C_{akj}, D_{akjm}, E_{akj}, F_{ak}$ are constants.
\end{lemma}

\begin{proof}
Calculating $F_{\mu\nu}^a$ yields
\begin{gather*}
F_{\mu\nu}^a=\sum_j\alpha_{\mu j}h'(r_j)(s_{a,1}d_{\nu,1}+s_{a,2}d_{\nu,2})e^{\sum_j h(r_j)}\\
-\sum_j\alpha_{\nu j}h'(r_j)(s_{a,1}d_{\mu,1}+s_{a,2}d_{\mu,2})e^{\sum_j h(r_j)}\\
-g\sum_{c>b}f_{abc}(s_{b,1}s_{c,2}-s_{b,2}s_{c,1})(d_{\mu,1}d_{\nu,2}-(d_{\mu,2}d_{\nu,1})e^{\sum_j 2h(r_j)}
\end{gather*}
The last term does not disappear, thus $F_{l,k}^a$ is of the form
\begin{gather*}
F_{lk}^a=\sum_ja_{lkj}h'(r_j)e^{\sum h(r_j)}+b_{lk}e^{\sum_j 2h(r_j)}
\end{gather*}
As $A^3_a=0$ the equations (2.22) reduce to (3.9). The terms in (3.9) are of 
the following form, the constants are complex numbers
\begin{gather*}
\partial^l F_{lk}^a=\sum_{j,l}a_{lkj}\alpha_{lj}h''(r_j)e^{\sum_j h(r_j)}
+\sum_{j,l,m}a_{lkj}\alpha_{lm}h'(r_j)h'(r_m)e^{\sum_j h(r_j)}\\
+\sum_{j,l}b_{lk}2h(r_j)\alpha_{lj}e^{\sum_j 2h(r_j)}
\end{gather*}
\begin{gather*}
\partial^3\partial^3 A_k^a=\sum_j c_{akj}h''(r_j)e^{\sum_j h(r_j)}
+\sum_j d_{akjm}h'(r_j)h'(r_m)e^{\sum_j h(r_j)}
\end{gather*}
\begin{gather*}
-gf_{abc}A^l_aF^a_{lk}=\sum_je_{lkj}h'(r_j)e^{\sum_j 2h(r_j)}+f_{lk}e^{\sum_j 3h(r_j)}
\end{gather*}
\end{proof}

\begin{lemma}
Let $h(r_j)=\beta^2r_j^2$ in  Lemma 4.1.
There are no solutions with $F_{ak}\not=0$ of the type in Lemma 4.1.
\end{lemma}
\begin{proof}
The term $e^{-2\beta^2\sum r_j^2}$ in (4.23) is not cancelled by anything.
\end{proof}

Lemma 4.2 shows that there are no interactions for solutions of the type (3.39)
but there could be solutions of as in Lemma 4.1 for some other $h(r_j)$.
In [4] a type of function is proposed as a solution for (4.23) but it is
not explicitly shown that such a solution exists. Even simple solutions of
the type (3.30) are not trivial and may give solutions that do not appear
in the free field case, i.e., when the coupling constant $g$ is zero. 

In any case, the largest group of solutions is surely
(3.30) since there we have a free function $h$, while if the structure 
constants appear in the equations, we get a nonlinear partial differential 
equation, at least as difficult or worse as in Lemma 4.1, 
which typically have fewer solutions. If more solution families are
found, correction terms can be calculated from the remaining parts of the 
path integral. Thus, we can make a non-trivial theory for a pure Yang-Mills 
Lagrangian and compute first order approximations. Let us mention that Quantum Electrodynamics (QED) and Quantum Chromodynamics (QCD) are not non-trivial
quantum field theories for the pure Yang-Mills Langangian but there the spinor
fields interacting with gauge fields create the interesting results.
As a conclusion, there is no good reason to exlude $A^a_{\mu}$ in (3.68)
or (3.74) in the path integral approach.

The other approach is axiomatic quantum field theory where Wightman's axioms, 
or something as strong, has especially been mentioned in the CMI problem.
We do not need to construct a theory filling axioms similar or stronger that
Wightman's but only to investigate if a theory filling such conditions should
include $A^a_{\mu}$ in (3.68), properly normalized, as a state. 
This involves showing two things. Firstly, that they can be included, and 
secondly that a theory that does not include them should be called trivial. 
Let us briefly go through Wightman's axioms. 
Wightman does not consider gauge 
fields at all and so we have to modify the axioms. 

Axiom I. {\it The states of a quantum field theory are normalised vectors in a separable Hilbers space, $\mathcal {H}$, two such that they differ by a complex phase giving raise to the same state.}  
If we normalize $A^a_{\mu}$ it is a normalized vector in a 
separable Hilbert space. Any states that differ by a complex phase give
rise to the same state. Thus, $A^a_{\mu}$ and $A^{a,R}_{\mu}$ give the same
state as these differ by a complex phase. Apparently we can compute the real 
Lagrangian $\mathcal{L}$ as this is what Axiom I seems to imply. Fortunately, 
$A^{a,R}_{\mu}$ is a solution to the real Euler-Lagrange equations.

Axiom II. {\it The space $\mathcal{H}$ 
carries a continuous unitary representation 
$(a,\Lambda)\mapsto U(a,\Lambda)$ of the restricted orthochronous Poincare group. In $\mathcal{H}$ there exists a vector, unique up to a phase, (called the vacuum state) that is invariant under all $U(a,\Lambda)$ and for all other vectors $\Psi\in \mathcal{H}$ the energy is positive.}
The only issue of concern here is that the energy of $A^a_{\mu}$, and
of $A^{a,R}_{\mu}$, is positive, which is shown in Theorem 3.14 for the metric
of (2.33). For Minkowski's metric it was shown the energy of the real 
Hamiltonian is zero. If also the imaginary part is zero this means that the 
vacuum is not unique. We may want to discard (3.68) in Minkowski's metric but
(3.74) gives positive energy and there is no reason to discard that field.   

Axiom IIIa. Deals only with the vacuum state and is of no concern to 
$A^a_{\mu}$ being an acceptable state or not. 

Axiom IIIb. {\it For any pair of vectors $\Phi$ and $\Psi$, the map $f\mapsto <\Phi,\phi(f)\Psi>$ is continuous.}
Here $\phi(f)=\int d^4x\phi(x)f(x)$ is the smeared field. The function
$f$ is tempered, i.e., belongs to $S$, the set of infinitely differentiable 
functions on $\mathbb{R}^4$ which vanish faster than any power of 
Euclidean distance. It guarantees that the integral converges. 
The inner product is given by an integral over $\mathbb{R}^4$. 
If one of the vectors $\Phi$ and $\Psi$ is $A_{\mu}^a$ and another one is not,
then fulfilment of the axiom depends on the other vector. 
If both vectors are of the type
$A_{\mu}^a$ then the map $f\mapsto <\Phi,\Psi(f)>$ is continuous. 

Axiom IV. {\it Suppose that $f,g\in S$ are such that ${\rm supp} f$ is 
space-like 
to ${\rm supp} g$; then $\phi(f)\phi(g)=\phi(g)\phi(f)$.}
This holds for the free field. $\phi=A_{\mu}^a$ is a solution to free field 
equations as the part with the structure constants cancels. 

These axioms do not have requirements that exclude $A_{\mu}^a$. Thus, 
$A_{\mu}^a$ can be included in a theory filling the axioms at if the
energy is positive. If the energy is zero, there is a problem in the theory.
In that case we may exclude $A_{\mu}^a$ in order to resolve the problem 
and fill Wightman's axioms, but it is a bit artificial way. 
The second part is to show that they must be included in a non-trivial theory.
The solutions $A_{\mu}^a$ are natural solutions to the Euler-Langange equations
and especially if the coupling constant $g=0$ or the group is $U(1)$ they are
among the possible solutions. They give arbitrarily small eigenvalues to the
Hamiltonian. While it may be possible to create a theory which does not include
these solutions and is still valid for $U(1)$ and $g=0$ cases, such a theory is
trivial since it can be made by the following trivial procedure. Take any 
theory filling the axioms. If it includes the states $A_{\mu}^a$, then exlude
all states that have these states as minimal solutions for the Lagrangian.
The resulting theory does not have these eigenstates for the Hamiltonian.
Indeed, we can make a theory with two states only, vacuum and 
an eigenstate of the Hamiltonian with a non-zero eigenvalue. 
Then all axioms are easily filled. 
A trick of this type can always be made and it avoids the essential problem of
showing that there is a mass gap and has no physical relevance. 
Thus, we should call trivial any quantum 
field theory that does not include the solutions of the type $A_{\mu}^a$ if 
(more accurately, as) they can be included. 

There are two manuscripts [4],[5] arguing that a mass gap exists.
Both start by imposing the temporal, or Weyl, gauge $A^0=0$.
If we impose this gauge and then look at the boundary conditions, the solutions
(3.68) cannot be found. This is because when we localize the gauge field we need three linearly independent vectors $r_j$ in (3.33). As can be seen in the 
selected space gauge $A^3=0$, we only get two vectors for the non-gauged 
coordinates as is shown in Lemma 3.6. The third vector must be obtained from 
the gauged coordinate. Had we gauged time, then the equations in (3.34) would
show that $x_0$ is limited, as now is $x_3$, while $x_1$ and $x_2$ would
be linearly dependent on $x_3$. Then the field would not be integrable over
the space coordinates, while it would be limited in time. Instead of fixing
the gauge first, we must first look at the boundary conditions. This shows 
that the temporal gauge is not the correct choice, while a space gauge can 
work.

\section{Final comments of the CMI Millennium Prize problem}

The CMI problem setting called for mathematical clarity to the area of 
gauge fields. Much of this lack of clarity has traditionally been caused by
mathematical unclarities in the path integral method. 
Everything is formulated in simple lemmas which are given proofs.
This does not imply that the lemmas are considered new, it is only for clarity.
The presentation of Yang-Mills fields follows the approach in [6].

There are some final words about the clarity of the CMI problem 
statement itself. 

The problem statement does not specify whether the gauge group should be local or global, and [5] understands that it is global gauge group. It probably must
be local gauge group since the relevant issues arise from local gauge invariance. However, this should have been stated.

The metric in the CMI problem setting is unclear.
Minkowski's metric (2.9) is the correct choice for quantum field theory
but the CMI problem setting only mentions $\mathbb{R}^4$ and
the expert's explanation in [2], referred to in (2.38), seems to point 
to the Euclidean metric and to real curvature. We can present the
results in $\mathbb{R}^4$ with the Euclidean metric also. 
The convenient way to do it is to use the negative definite metric (2.33).
In Section 3 we have used (2.11). It is still valid for the metric (2.33), as
(2.5) is the definition and in (2.11) we have simply multiplied (2.6) by 
$g_{\mu\beta}g_{\nu\beta}$. 
We have also used (2.22). In the derivation of (2.22) we have kept the metric
explicitely and not used the values of $g_{\mu\nu}$ from (2.9). Thus, (2.22)
is also valid for (2.33). There are no raising or lowering $x_0$ indices in
Lemmas 3.1-3.9, thus they stay valid. In Lemma 3.10 we use (2.27) but give the
result also for the metric in (2.33). Lemma 3.11 has no changes. 
There is derivation with respect to $x_0$ in Lemma 3.12 but the conclusions
remain since they are caused by the disappearance of the integral (3.61) as
is mentioned in to proof.
It follows that Theorem 3.14 holds also for the metric (2.33) for some other
constant $C$. 
It is assumed that the fields can be complex as it is the situation in the 
physical problem and the Hodge star operation is defined for differential 
forms in complex manifolds. But as it is unclear in (2.38) and in the problem
setting the calculations were done for the real part of the curvature covering
the possibility that the problem statement implies real fields. 
It would have been much clearer if the CMI problem statement
had stated if Minkowski's metric is assumed, and if the fields
are complex or real. 

Referring to axiomatic field theory by mentioning axioms that do not as such
apply to gauge fields, use of words such as non-trivial, etc. would make any 
positive solutions to the CMI problem difficult to argue. This would not
be an issue if proposed solutions to the CMI problems would be positively
received and carefully reviewed. It would be an issue if the opposite were 
the case. 

The results of this article are easier to verify:

It seems that the CMI problem refers to Euclidean metric and real curvature.
As there is no minus sign in (2.34) and (2.38) while there is one in (2.1) it
seems that the metric is as in (2.33). In this case there is no mass gap since
we can by selection of $\beta$ in (3.73) make the eigenvalue of the Hamiltonian
as small as desired.

If the problem means Minkowski's metric and real fields, then the gauge field 
in (3.74) shows that there is no mass gap. However, the field (3.68) gives
zero energy and indicates that vacuum is not unique and Wightman's axioms 
cannot be filled. We may want to exclude the field (3.68) in this case but
there is no good reason for excluding it.

If the problem means complex fields in either metric, the conclusions 
are the same.

The results presented here should not be called a trivial free field theory. 
The coupling constant $g$ is not set to zero. The solutions that have been 
found are of such a type that the part with structure constants cancel. 
As Lemmas 4.1 and 4.2 indicate, nontrivial results can be found starting 
from the solutions in (3.30) and (3.39). Localization of the field in space
is not trivial and in general this word should be avoided if clarity is 
desired because clarity is best achieved by writing down all steps. This 
article may be correctly called elementary and easy, but not trivial.

\end{document}